\documentclass[12pt]{article}
\usepackage{amsmath,amssymb,amsthm,anysize,booktabs,tikz}
\usepackage{fancyhdr}
\usepackage{amsfonts,yfonts}
\usepackage{latexsym}
\usepackage{enumerate}
\usepackage{csquotes}
\usepackage{lineno,xcolor}
\usepackage[normalem]{ulem}
\usepackage{scalefnt}

\usepackage{geometry}
\usepackage{lipsum}
\usepackage{tikz}
\usetikzlibrary{arrows}
\usetikzlibrary{positioning}

\usepackage{lineno,xcolor}

\usepackage{hyperref}
\hypersetup{
    colorlinks,
    linkcolor={red!60!black},
    citecolor={green!60!black},
    urlcolor={blue!60!black}
}

\makeatletter

\newtheorem*{maintheorem*}{Main Theorem}
\newtheorem{theorem}{Theorem}[section]
\newtheorem{proposition}[theorem]{Proposition}
\newtheorem{corollary}[theorem]{Corollary}
\newtheorem{lemma}[theorem]{Lemma}

\newtheorem*{theorem*}{Theorem}
\newtheorem*{conjecture*}{Conjecture}
\newtheorem{conjecture}{Conjecture}

\theoremstyle{remark}
\newtheorem{remark}[theorem]{Remark}

\newtheorem*{example*}{Example}
\newcommand{\bvec}[1]{\boldsymbol{#1}}

\renewcommand\le{\leqslant}

\renewcommand\j{\mathbf j}

\newcommand\ba{\mathbf a}

\newcommand\m{\mathbf m}
\newcommand\N{\mathrm N_{q^2/q}}
\newcommand\n{\mathbf n}
\newcommand\p{\mathbf p}
\newcommand\q{\mathbf q}
\renewcommand\r{\mathbf r}

\renewcommand\v{\mathbf v}

\newcommand\X{\mathbf X}
\newcommand\x{\mathbf x}
\newcommand\y{\mathbf y}

\newcommand\cD{\mathcal D}
\newcommand\cH{\mathcal H}
\newcommand\cO{\mathcal O}
\newcommand\cP{\mathcal P}
\newcommand\cX{\mathcal X}
\newcommand\cV{\mathcal V}
\newcommand\cU{\mathcal U}
\newcommand\cR{\mathcal R}
\newcommand\cS{\mathcal S}
\newcommand\cQ{\mathcal Q}
\newcommand\bR{{\mathbb R}}

\newcommand\GF{{\rm GF}}
\newcommand\PG{{\rm PG}}
\newcommand\PGU{{\rm PGU}}

\def\sgn{{\rm sgn}}

\newcommand\Bil{{\rm Bil}}
\newcommand\rank{{\rm rank}}
\newcommand\Tr{{\rm Tr}_{q^2/q}}
\def\mod{{\rm mod} }
\def\diag{{\rm diag} }

\newcommand\comment[1]{}

\newcommand*{\shifttext}[2]{
  \settowidth{\@tempdima}{#2}
  \makebox[\@tempdima]{\hspace*{#1}#2}
}

\newcommand\redout{\bgroup\markoverwith
{\textcolor{red}{\rule[.4ex]{2pt}{0.8pt}}}\ULon}

\allowdisplaybreaks

\title{Pseudo-ovals of elliptic quadrics\\ as Delsarte designs of association schemes}

\author{John Bamberg\\
\small {\tt john.bamberg@uwa.edu.au}\\
\small School of Physics, Mathematics and Computing\\[-0.8ex]
\small University of Western Australia\\[-0.8ex]
\small 35 Stirling Highway \\[-0.8ex]
\small Perth WA 6009, Australia\\
\and
Giusy Monzillo\footnote{The research was supported by the Italian National Group for Algebraic and Geometric Structures and their Applications (GNSAGA-INdAM). } \\
\small {\tt giusy.monzillo@unibas.it}\\[0.8ex]
\small Dipartimento di Matematica, Informatica ed Economia\\[-0.8ex]
\small Universit\`a degli Studi della Basilicata\\[-0.8ex]
\small Viale dell'Ateneo Lucano 10 \\[-0.8ex]
\small 85100 Potenza, Italy\\
\and
Alessandro Siciliano$^*$ \\
\small{\tt alessandro.siciliano@unibas.it}\\[0.8ex]
\small Dipartimento di Matematica, Informatica ed Economia\\[-0.8ex]
\small Universit\`a degli Studi della Basilicata\\[-0.8ex]
\small Viale dell'Ateneo Lucano 10 \\[-0.8ex]
\small 85100 Potenza, Italy\\
}

\begin{document}

\date{}
\maketitle 
\thispagestyle{fancy}
\fancyhf{}
\renewcommand{\headrulewidth}{0pt}
%
\begin{abstract}
A {\em pseudo-oval} of a finite projective space over a finite field of odd order $q$ is a configuration of equidimensional subspaces that is essentially equivalent to a translation generalised quadrangle of order $(q^n,q^n)$ and a Laguerre plane of order $q^n$ (for some $n$). In setting out a programme to construct new generalised quadrangles, Shult and Thas \cite{st} asked whether there are pseudo-ovals consisting only of lines of an elliptic quadric ${Q}^-(5,q)$, non-equivalent to the \emph{classical example}, a so-called \emph{pseudo-conic}. To date, every known pseudo-oval of lines of ${Q}^-(5,q)$ is projectively equivalent to a pseudo-conic. Thas \cite{thas} characterised pseudo-conics as pseudo-ovals satisfying the \emph{perspective} property,
and this paper is on characterisations of pseudo-conics from an algebraic combinatorial point of view. In particular, we show that pseudo-ovals in $Q^-(5,q)$ and pseudo-conics can be characterised as certain Delsarte designs of an interesting five-class association scheme. These association schemes are introduced and explored, and we provide a complete theory of how pseudo-ovals of lines of $Q^-(5,q)$ can be analysed from this viewpoint.
\end{abstract}

{\it Keywords: Association scheme, Elliptic quadric, Pseudo-oval}             

{\it Math. Subj. Class.: 05E30, 51A50}

\section{Introduction}

A {\em pseudo-oval} of the finite projective space $\PG(3n-1,q)$ is a set of $q^n+1$ subspaces, each of dimension $n-1$, such that any three distinct elements of the set span the whole space. Such configurations are essentially \emph{equivalent} to translation generalised quadrangles of order $(q^n,q^n)$ \cite{pt}. For $q$ odd, a pseudo-oval of $\PG(3n-1,q)$ is equivalent to a Laguerre plane of order $q^n$. 

The {\em classical example} can be constructed in the following way. If we consider the $\GF(q^n)$-vector space underlying $\PG(2,q^n)$ as a $\GF(q)$-vector space,  each point of $\PG(2,q^n)$ becomes an $(n-1)-$subspace of $\PG(3n-1,q)$.  In particular, the $q^n+1$ subspaces corresponding to the points of a non-degenerate conic of $\PG(2,q^n)$ form a pseudo-oval, known as \emph{pseudo-conic}.  Thas \cite{thas} characterised pseudo-conics as pseudo-ovals of $\PG(3n-1,q)$, $q$ odd, satisfying the \emph{perspective} property (see Section \ref{sec_2} for more details).

Let $q$ be odd. For $n$ even any pseudo-conic belongs to an elliptic quadric $Q^-(3n-1,q)$, and for $n$ odd any pseudo-conic belongs to a non-degenerate parabolic quadric of $\PG(3n-1,q)$ \cite{st}. For $q$ even a pseudo-oval is never contained in  a non-degenerate quadric \cite{thas}.
 In the quest to construct new generalised quadrangles, Shult and Thas \cite{st} asked whether there are pseudo-ovals consisting only of lines of an elliptic quadric ${Q}^-(5,q)$, non-equivalent to the classical example. 
 To date, every known pseudo-oval of lines of ${Q}^-(5,q)$ is projectively equivalent to a pseudo-conic. Indeed, the discovery of a new pseudo-oval of $Q^-(5,q)$ would result in a new generalised quadrangle and  new Laguerre plane.

Under the Klein correspondence, pseudo-ovals contained in $Q^-(5,q)$ are mapped onto \emph{special sets} of $H(3,q^2)$ \cite{shult}. A {\em special set} of the Hermitian surface $H(3,q^2)$ is a set $\cS$ of $q^2+1$ points such that any  point of $H(3,q^2)$ not in $\cS$ is orthogonal  to 0 or 2 points of $\cS$ \cite{shult}.
 From a result by De Soete and Thas \cite{dst}, $q$ is necessarily odd. 
 Bader, O'Keefe, Penttila in \cite{bokp} and, independently,   Shult  in \cite{shult} constructed an example of a special set of $H(3,q^2)$. This consists of the $q^2+1$ points of an elliptic quadric over $\GF(q)$ which is the complete intersection of $H(3,q^2)$ with a hyperbolic quadric of $\PG(3,q^2)$ whose polarity commutes with the given unitary one. The special sets in this class are called {\em  of CP-type} \cite{ckm}. Theorem 3.1 in \cite{ckm} gives a characterisation of special sets of CP-type in terms of the unitary form defining $H(3,q^2)$. 
By \cite[Theorem 2.1]{cmp}, a special set of CP-type  corresponds to a pseudo-conic.

This paper is on characterisations of pseudo-conics from an algebraic combinatorial point of view. In particular, we will show (see Theorem \ref{lem_13}) that pseudo-ovals  and pseudo-conics in $Q^-(5,q)$ can be characterised as certain Delsarte designs of an interesting five-class association scheme. These association schemes are introduced and explored, and we provide a complete theory of how pseudo-ovals of lines of $Q^-(5,q)$ can be analysed from this viewpoint. 

The paper is organised as follows. Section \ref{background} contains some notation and introductory material on projective geometry, classical polar geometries and association schemes. In Section \ref{sec_2} we investigate the perspective property for lines of $Q^-(5,q)$. By representing subspaces of $\PG(5,q)$ in a matrix form, we give an algebraic characterisation of being in perspective for a triple of lines of $Q^-(5,q)$ (Proposition \ref{prop_2}). This allows us to translate the above algebraic condition in terms of a (local) geometric property involving certain configurations arising from non-degenerate hyperplanes (Proposition \ref{lem_12}).
In Section \ref{sec_3} an imprimitive five-class association scheme is constructed on certain points of $H(3,q^2)$. The relations of the scheme are defined by considering a function, introduced by Shult  in \cite{shult}, associated with the hermitian form of $H(3,q^2)$. As a by-product, the study of its quotient scheme produces a strongly regular graph isomorphic to the bilinear forms graph $\Bil_{2}(q)$ (Proposition \ref{prop_6}).
Section \ref{sec_4} is the real core of the whole paper: the opening theorem, providing the link between Shult's function and the property to be in perspective for the lines of $Q^-(5,q)$, allows us to consider pseudo-ovals of $Q^-(5,q)$, as well as  pseudo-conics, as subsets of a five-class association scheme on certain lines of $Q^-(5,q)$, isomorphic to the scheme explored in Section \ref{sec_3}. 
In this setting, from comparing the characteristic vector of a pseudo-oval with the common eigenspaces of the scheme,  we provide a characterisation of pseudo-conics in terms of the configurations introduced in Section \ref{sec2U}. Finally, Section \ref{sec_7}  contains some computational results leading to the conjecture that every pseudo-oval in $Q^-(5,q)$ is a pseudo-conic, for all $q$ odd.

\section{Background theory}\label{background}

For any given $n-$dimensional vector space $V=V(n,F)$ over the field $F$,  the {\em projective geometry} defined by $V$ is the partially ordered set of all subspaces of $V$, and it will be denoted by  $\PG(V)$. Two elements of $\PG(V)$ are said to be {\em disjoint} or {\em skew} if they intersect in the zero vector.  In order to simplify notation, for each proper subspace $U$ of $V$, that is an element of $\PG(V)$, we will use the same letter for the projective geometry defined by $U$. If $S \subset V$, we use  $\langle S \rangle$ to denote the subspace spanned by $S$. 

 If $F$ is the finite field $\GF(q)$ with $q$ elements, then we may write $V=V(n,q)$ and $\PG(n-1,q)$ instead of $\PG(V)$. The 1-dimensional subspaces are called {\em points}, the 2-dimensional subspaces are called {\em lines}, the 3-dimensional subspaces are called {\em planes}, and  the $(n-1)-$dimensional subspaces are called {\em hyperplanes} of $\PG(V)$. If $V$ is endowed with a non-degenerate alternating, quadratic or Hermitian form of Witt index $m$, the set of totally isotropic (or totally singular, in case of a quadratic form) subspaces of $V$ is a {\em polar geometry of rank $m$} of $\PG(V)$, which is called {\em symplectic, orthogonal}  or {\em unitary}, respectively. When $n=2r$, the vector space $V$ has precisely two (non-degenerate) quadratic forms, and they differ by their Witt index. It can be $r-1$ or $r$, and the quadratic form is \emph{elliptic} or \emph{hyperbolic}, respectively.
It is customary to set $\sgn(Q)=-$ in the  former case, and  $\sgn(Q)=+$ in the latter.  In terms of the associated projective geometry $\PG(V)$, the orthogonal polar geometry arising from an elliptic (resp. hyperbolic) quadratic form is known as an {\em elliptic} (resp. {\em hyperbolic})  {\em quadric} of $\PG(V)$, and it is denoted by $Q^-(n-1,q)$ (resp. $Q^+(n-1,q)$).
Our principal reference on projective geometries and polar geometries  is \cite{taylor}.

 Association schemes are important objects in algebraic combinatorics that generalise distance-regular graphs, linear codes, and combinatorial designs. As we shall see, the theory of association schemes can be a powerful tool when applied to some problems in finite geometry. An association scheme $\mathfrak X=(X,\{R_i\}_{0\le i\le d})$ is a set of \emph{vertices} $X$ and binary {\em relations} $R_i$ on $X$ satisfying the following:
\begin{enumerate}
    \item $R_0$ is the diagonal relation, that is, $R_0=\{(x,x):x \in X\}$.
    \item $\{R_i\}$ is closed under taking the opposite relation; that is, $R_j^*:=\{(x,y)\colon (y,x)\in R_j\}$ is in $\{R_i\}$, for each $j$.
    \item For each $i,j,k\in \{0,\ldots, d\}$, there exist constants $p_{i,j}^k$, such that if $(x,y)\in R_k$, then there are $p_{i,j}^k$ vertices $z$ such that $(x,z)\in R_i$ and $(z,y)\in R_j$. The $p_{i,j}^k$ are called {\em intersection numbers}.
\end{enumerate}
We will say that the association scheme is \emph{symmetric} if each relation is equal to its opposite. Let $\mathfrak X=(X,\{R_i\}_{0\le i\le d})$ be an association scheme with $d$ classes. For $0\le i\le d$, let $A_i$ be the adjacency matrix of the relation $R_i$, and $E_i$ the $i-$th primitive idempotent of the Bose-Mesner algebra of $\mathfrak X$ which projects onto the $i-$th maximal common eigenspace of $A_0,\ldots,A_d$. The  matrices $\cP$ and $\cQ$ defined by
\[
(A_0\ A_1\ \ldots \ A_d)=(E_0\ E_1\ \ldots \ E_d)\cP
\]
and 
\[
(E_0\ E_1\ \ldots \ E_d)=|X|^{-1}(A_0\ A_1\ \ldots \ A_d)\cQ
\]
are the {\em first} and the {\em second eigenmatrix} of $\mathfrak X$, respectively.  The reader is referred to \cite{bi,cam,del} for additional information on association schemes.

\section{Investigating the perspective property}\label{sec_2}
 In $V=V(6,q^2)$  consider the 6-dimensional $\GF(q)-$subspace 
 \[
 \widehat{V}=\{(x,x^q,y,y^q, z,z^q):x,y,z\in\GF(q^2)\}.
 \]
 Let $\PG(\widehat V)$ be the projective geometry defined by $\widehat V$. 
 For any vector $(x,x^{q},y,y^{q},z,z^q)\in\widehat V$ we will use the  short-hand notation $(x,y,z)_2$.

We consider the hyperbolic quadric $Q^+(5,q^2)$ of $\PG(5,q^2)$ (known as {\em Klein quadric})   defined by the (non-degenerate) quadratic form
$
Q(\X)=-X_1X_6- X_2X_5+X_3X_4$
on $V(6,q^2)$. For any given $v=(x,y, z)_2\in\widehat V$,  
\begin{equation}\label{eq_18}
\widehat{Q}(v)=Q|_{\widehat V}(v)=- xz^q-x^qz+y^{q+1}.
\end{equation}
It turns out that    $\widehat{Q}$  is a non-degenerate quadratic form of rank 2 on $\widehat V$ with associated symmetric form 
\begin{equation*}\label{eq_8}
\widehat{\bf b}(v,v')=-xz'^q-x^qz'+yy'^q+y^qy'-zx'^q -z^qx'.
\end{equation*}
 Therefore, $\widehat{Q}$  gives rise to an elliptic quadric  $Q^-(5,q)$ of $\PG(\widehat{V})$ embedded in $Q^+(5,q^2)$.
 For any  subspace  $W$ of $\widehat{V}$, set 
\[
W^\perp=\{v \in \widehat{V}:  \widehat{\bf b}(v,u)=0, \mathrm{\ for\ all\ }u\in W \}.
\] 
In the following, $\Tr$ and ${\N}$ will denote the {\em relative trace} and {\em norm} functions from $\GF(q^2)$ onto $\GF(q)$.
Let $F$ be a $\GF(q)-$linear transformation from $\GF(q^2)$	to itself. Then, $F$ can be represented by a unique polynomial over $\GF(q^2)$ of type $F(x)=ax+bx^q$. Such a polynomial is called a $q-${\em polynomial over }$\GF(q^2)$ \cite[Chapter 3]{ln}. 
The trivial $q-$polynomial will be denoted by $I$. The {\em adjoint} of a linearised polynomial $F(x)=ax+bx^q$, 	with respect to the symmetric bilinear form $(a,b)\rightarrow \Tr(ab)$, is given by $F^*(x)=ax+b^qx^q$. 

 In $\widehat V$, any line $l$ is written as 
\[
l=\{(F_0(x),F_1(x),F_2(x))_2:x \in\GF(q^2)\}
\]
where $F_0,F_1,F_2$  are $q-$polynomials over $\GF(q^2)$; for short, we will write  $l=L(F_0,F_1,F_2)$. The triple $(F_0,F_1,F_2)$ is determined by $l$ up to a right factor of proportion, which is a non-singular $q-$polynomial.
 Since a 4-dimensional subspace of $\widehat V$ is a 2-dimensional subspace in the dual space $\widehat V^*$, any such a subspace $T$ can be represented  by three  $q-$polynomials $H_0,H_1,H_2$ over $\GF(q^2)$. A way to write equations for $T$ is the following. Fix an element  $\theta\in \GF(q^2)\setminus\GF(q)$.
Let $\cH_i$ be the $2\times 2$ Dickson matrix\footnote{The \emph{Dickson matrix} of the $q-$polynomial $\sum_{i=0}^{n-1}a_ix^{q^i}\in\GF(q^n)[x]$ is
$\begin{pmatrix}
a_0 & a_1 & \cdots & a_{n-1}\\
a_{n-1}^q & a_0^q & \cdots & a_{n-2}^q\\
\vdots & \vdots & \vdots & \vdots \\
a_1^{q^{n-1}} & a_2^{q^{n-1}} & \cdots & a_0^{q^{n-1}}\\
\end{pmatrix}$.} associated with $H_i$,  $i=1,2,3$, then $T$ has  equations
\[
\begin{array}{lcl}
\begin{pmatrix}
1 & 1 \\
\theta &\theta^q
\end{pmatrix}
\begin{pmatrix}
\cH_0 & \cH_1 & \cH_2
\end{pmatrix}
\begin{pmatrix}
x \\ x^q \\ y \\ y^q \\z \\ z^q
\end{pmatrix}
=0
\end{array}.
\]

For short,  write $T=\pi(H_0,H_1,H_2)$. The triple $(H_0,H_1,H_2)$ is determined by $T$ up to a left factor of proportion, which is a non-singular $q-$polynomial. It is easy to check that a line $L(F_0,F_1,F_2)$ is contained in the subspace $\pi(H_0,H_1,H_2)$ if and only if 
\begin{equation*}\label{eq_6}
H_0\circ F_0+H_1\circ F_1+H_1\circ F_1=0;
\end{equation*}
where $H\circ F$ is the $q-$polynomial $H(F(x))\,\mod\,(x^{q^2}-x)$.

\comment{
in the matrix form
\[
l=\left\{\begin{pmatrix}
F_0 \\
F_1 \\
 F_2
\end{pmatrix}
\begin{pmatrix}
x\\
x^q
\end{pmatrix}
:x \in \GF(q^2)
\right\},
\]
where $F_0,F_1,F_2$  are two-by-two Dickson matrices and $\begin{pmatrix}
F_0 \\
F_1 \\
 F_2
\end{pmatrix}$ has rank 2; we will use the  short-hand notation $l=L(F_0,F_1,F_2)$. The triple $(F_0,F_1,F_2)$ is determined by $l$ up to a non-singular Dickson matrix on the right side.
}

The line $l=L(F_0,F_1,F_2)$ is totally singular with respect to the symmetric form $\widehat{\bf b}$ if and only if 
\begin{equation}\label{eq_4}
 F_2^*\circ K\circ F_0-F_1^*\circ K\circ F_1+ F_0^*\circ K\circ F_2=0,
\end{equation}
where  $K(x)=x^q$ (note that $K^*=K$). Let $F_i(x)=f_ix+g_ix^q$, $i=0,1,2$. Then, Eq.\,\eqref{eq_4} is equivalent to
\begin{equation}\label{eq_5}
\left\{
\begin{array}{lcl}
f_2g_0^q+ f_0 g_2^q & = & f_1g_1^q\\[.1in]
f_0 f_2^q+f_0^qf_2+ g_0 g_2^q+g_0^qg_2 & = & f_1^{q+1}+g_1^{q+1}
\end{array}
\right..
\end{equation}
Let $l_1$, $l_2$, $l_3$ be mutually skew  lines of $\PG(\widehat V)$, $T_i$ be a 4-dimensional space containing $l_i$ but skew to $l_j$ and $l_k$, and  $s_k=T_i\cap T_j$, with $\{i,j,k\}=\{1,2,3\}$.  The space spanned by $s_i$ and $l_i$ will be denoted by $\Sigma_i$, with $i=1,2,3$. If $\Sigma_1$, $\Sigma_2$ and $\Sigma_3$ have non-trivial intersection, then $\{l_1, l_2, l_3\}$ and $\{T_1,T_2,T_3\}$ are said to be {\em in semi-perspective}; if $\Sigma_1$, $\Sigma_2$ and $\Sigma_3$ share  a line, then $\{l_1, l_2, l_3\}$ and $\{T_1,T_2,T_3\}$ are said to be {\em in perspective}. For our aims, if  $l_1$, $l_2$, $l_3$ are lines of $Q^-(5,q)$, we set  $T_i=l_i^\perp$ and we will simply say that $l_1, l_2, l_3$ are in semi-perspective or perspective. 

Since for $q$ even pseudo-ovals are never contained in  an orthogonal polar geometry \cite{thas}, from now on we assume $q$ is  odd.

 Let $\theta\in\GF(q^2)\setminus\GF(q)$ be taken such that $\theta^2=\xi$ with $\xi$ a non-square in $\GF(q)$, i.e., $\theta^q=-\theta$. 

The following result translates  \cite[Theorem 5.1]{thas} in terms of the projective geometry $\PG(\widehat V)$. 
\begin{proposition}\label{prop_2}
Consider the three lines of the $Q^-(5,q)$, arising from $\widehat Q$,
\[
l=L(I,0,0), \ \ \ m=L(0,0,I), \ \ \  n=L(F_0,F_1,F_2),
\]
with $F_i(x)=f_ix+g_ix^q$, $i=0,1,2$, spanning the whole space. Then, $l,m,n$ 
are in perspective if and only if  
$f_0^qf_2+g_0g_2^q \in\GF(q)$.
\end{proposition}
\begin{proof}
As $\langle l,m \rangle=\{(x,0,z)_2:x,z \in\GF(q^2)\}$ and $n$ trivially intersects $\langle l,m \rangle$, then  $F_1$ is invertible.
We set $T_1=l^{\perp}$, $T_2=m^{\perp}$, $T_3=n^{\perp}$. Straightforward calculation yields 
\[
T_1=\pi(0,0,I),\quad  T_2=\pi(I,0,0),\quad  T_3=\pi(F_2^*\circ K,-F_1^*\circ K,F_0^*\circ K).
\]

Further,
\[
\begin{array}{ccccl}
s_3 & = & T_1\cap T_2& = & L(0,I,0), \\[.05in]
s_2 & = & T_1\cap T_3& = & L(I,(K\circ F_2 \circ F_1^{-1}\circ K)^*,0),\\[.05in]
s_1 & = & T_2\cap T_3& = & L(0,(K\circ F_0 \circ F_1^{-1}\circ K)^*,I).
\end{array}
\]

Now we want to write $\Sigma_1 = \langle l,s_1 \rangle$ and $\Sigma_2 = \langle m,s_2 \rangle$ in the form $\pi(H_0,H_1,H_2)$. To do this, we solve the linear system
\[
\begin{array}{lcl}
\begin{pmatrix}
1 & 1 \\
\theta &\theta^q
\end{pmatrix}
\begin{pmatrix}
\cH_0 & \cH_1 & \cH_2
\end{pmatrix}
\begin{pmatrix}
x \\ x^q \\ y \\ y^q \\z \\ z^q
\end{pmatrix}
=0
\end{array},
\]
where, in turn, we substitute in the coordinates of four linearly independent vectors of $\Sigma_i$, $i=1,2$. Consequently, 
\[
\Sigma_1 =\pi(0,I,-(K\circ F_0\circ F_1^{-1}\circ K)^*), \ \ \ 
\Sigma_2 =\pi(-(K\circ F_2\circ F_1^{-1}\circ K)^*,I,0).
\]
Since $\Sigma_3=\langle n,s_3 \rangle=\{(F_0(x),y,F_2(x))_2:x,y\in\GF(q^2)\}$, by imposing that the generic point of $\Sigma_3$ belongs to $\Sigma_1$ as to $\Sigma_2$,   the points of $\Sigma_1\cap\Sigma_2\cap\Sigma_3$ are obtained by solving the following system of linear equations
\begin{equation*}
\left\{
\begin{array}{lcl}
 y-(F_2^*\circ K\circ F_0 \circ F_1^{-1}\circ K)^*(x) & = & 0\\[.1in]
y-(F_0^*\circ K\circ F_2 \circ F_1^{-1}\circ K)^*(x) & = & 0.
\end{array}
\right.
\end{equation*}
This yields
\[
(F_2^*\circ K\circ F_0- F_0^*\circ K\circ F_2)(x) = 0.
\]
From Eq.\,\eqref{eq_4}, we get
\begin{equation}\label{eq_4bis}
( 2F_2^*\circ K\circ F_0-F_1^*\circ K\circ F_1)(x) = 0.
\end{equation}
Note that $(F_2^*\circ K\circ F_0)(x)=(f_0g_2^q+f_2g_0^q)x+(f_0^qf_2+g_0g_2^q)x^q$ and $(F_1^*\circ K\circ F_1)(x)=2f_1g_1^qx+(f_1^{q+1}+g_1^{q+1})x^q$. Therefore, 
{\small
\begin{align*}
(2F_2^*\circ K\circ F_0-F_1^*\circ K\circ F_1)(x) & =  2(f_0g_2^q+f_2g_0^q-f_1g_1^q)x+
[2(f_0^qf_2+g_0g_2^q)-f_1^{q+1}-g_1^{q+1}]x^q\\
 & = (f_0^qf_2-f_0f_2^q+g_0g_2^q-g_0^qg_2)x^q,
\end{align*}}
by Eq.\,\eqref{eq_5}. Consequently, $\Sigma_1\cap \Sigma_2\cap\Sigma_3=\{0\}$ if and only if 0 is the unique solution of \eqref{eq_4bis} if and only if $f_0^qf_2+g_0g_2^q \not\in \GF(q)$. Similarly,   $\Sigma_1\cap \Sigma_2\cap\Sigma_3$ is a line if and only if $f_0^qf_2+g_0g_2^q \in\GF(q)$.
\end{proof}
\begin{remark}\label{rem_1}
In the proof of the previous result we used that $F_1$ is invertible. This property holds also for $F_0$ and $F_2$, because $n$ trivially intersects $m$ and $l$. 
\end{remark}

\subsection{Construction of the subsets of type $\cU_{p_1,p_2}$}\label{sec2U} 

As above, let $Q^-(5,q)$ be the elliptic quadric of $\PG(\widehat V)$ defined by $\widehat Q$, and fix a totally singular line $l$. For any given non-degenerate hyperplane $\Pi$  not containing $l$, let $B=l\cap \Pi$ and  $\sigma$ be the line $\langle l, \Pi^\perp\rangle \cap \Pi$. As $l\subset B^\perp$ and $B\in\Pi$ then $\langle l,\Pi^\perp \rangle\subset B^\perp$, hence $\sigma\subset B^\perp\cap\Pi$.
 In particular, $\sigma$ corresponds to an internal point for the non-singular conic $(B^\perp\cap\Pi\cap Q^-(5,q))/B$ of the quotient space  $(B^\perp\cap\Pi)/B$. To see this we observe that $\sigma^\perp=\langle l^\perp\cap\Pi, \Pi^\perp\rangle $ shares with the quadratic cone $B^\perp\cap\Pi\cap Q^-(5,q)$ just the point $B$. Therefore, if $\sigma$ corresponded to an external point, $\sigma^\perp$ would have two generators in common with the cone, which is a contradiction.
Then, for any given  totally singular line $p_1$ lying in $\Pi$ and passing through $B$, the plane 
$\langle p_1, \sigma \rangle$  meets ${Q}^-(5,q)$ in a further line $p_2$. Let $\cO_{i}$, $i=1,2$, be the  totally singular lines in $\Pi$ intersecting $p_i$, but not at $B$.
 We set $\cU_{p_1,p_2}=\cO_{1}\cup \cO_{2}$, and  $\cU_{p_1,p_2}$ is said to be {\em constructed on the flag} $(B,l)$.

By the reasoning above, it is now evident that the line $\sigma$ induces  an involution $\tilde\sigma$ on the generators of the quadratic cone $B^\perp\cap\Pi\cap Q^-(5,q)$. 
\begin{lemma}\label{lem_8}
Let $B=\langle (1,0,0)_2 \rangle$ and $l=L(I,0,0)$. 
\begin{enumerate}[(i)]
\item A non-degenerate hyperplane $\Pi$ through $B$ not containing  $l$ has equation
\[
\Pi: \theta (X-X^q)+\beta^{q}Y+\beta Y^q-\alpha^qZ-\alpha Z^q=0
\]
for all $\alpha,\beta\in\GF(q^2)$,  such that $\beta^{q+1}-\theta (\alpha^q-\alpha)\neq 0$;
\item the generators of the quadratic cone $B^\perp\,\cap\,\Pi\,\cap\, Q^-(5,q)$ have the form $l_y=L(F_0,F_1,F_2)$, where 
\[
F_0(x)=(2\xi-y^{q+1})x+(2\xi+y^{q+1})x^q, \ \ \ F_1(x)=2\theta y(x-x^q), \ \ \  F_2(x)=2\xi (x-x^q),
\]
for all $y\in \GF(q^2)$ such that 
\begin{equation}\label{eq_10}
 y^{q+1}-(\beta^q y+\beta y^q)+\theta(\alpha^q-\alpha)=0.
\end{equation}
\item  the involution induced by $\sigma=\langle l,\Pi^\perp \rangle\cap\Pi$ on the $l_y$'s is
\[
\begin{array}{ccc}
l_y & \overset{\tilde\sigma}\longmapsto  & l_{2 \beta-y}\, .
\end{array}
\] 
\end{enumerate}
\end{lemma}

\begin{proof}\leavevmode
(i) Under the polarity $\perp$ of $\PG(\widehat V)$ associated with $\widehat{\bf b}$,  a non-degenerate hyperplane $\Pi$ not containing  $l$ corresponds to a non-singular point $P\in B^\perp\setminus l^\perp$. Such a point has the form $P=\langle (\alpha,\beta,\theta )_2 \rangle$, with $\alpha,\beta\in\GF(q^2)$,  such that $\beta^{q+1}-\theta (\alpha^q-\alpha)\neq 0$.

(ii) In order to find the totally singular lines through $B$, we consider the restriction of $\widehat Q$ on  the 4-dimensional subspace $\Sigma=\{(\theta a,y,\theta b)_2:a,b\in \GF(q),y\in\GF(q^2)\}$ of $B^\perp$ not on $B$. We get that these totally singular lines, apart from $l$, have the form $l_y=L(F_0,F_1,F_2)$, where 
\[
F_0(x)=(2\xi-y^{q+1})x+(2\xi+y^{q+1})x^q, \ \ \ F_1(x)=2\theta y(x-x^q), \ \ \  F_2(x)=2\xi (x-x^q),
\]
for all $y\in \GF(q^2)$. In particular, $l_y$ is in  $\Pi$ if and only if  Eq.\,\eqref{eq_10} holds.

(iii) By definition, $\sigma=\langle l, \Pi^\perp \rangle\cap \Pi$ is $L(F_0,F_1,F_2)$ where 
\[
F_0(x)=\left(\frac{2\beta^{q+1}}{\theta }+\alpha-2\alpha^q\right)x-\alpha x^q, \ \ \ F_1(x)=-\beta(x+x^q), \ \ \  F_2(x)=\theta  (x+x^q),
\]

Fix the points $R=\langle (\alpha^q-\frac{\beta^{q+1}}{\theta},\beta,\theta)_2 \rangle$  of $\sigma$ and $R_1=\langle (-\theta y_1^{q+1},2\xi y_1,2\xi \theta)_2 \rangle$  of $p_1=l_{y_1}$ with $y_1$ satisfying Eq.\,\eqref{eq_10}. Since $\sigma$ corresponds to an internal point for the non-singular conic $(B^\perp\cap\Pi\cap Q^-(5,q))/B$ of the quotient space  $(B^\perp\cap\Pi)/B$, the line $p_2$ is the unique totally singular line $l_{y}$, $y\neq y_1$, intersecting the line $\langle R,R_1 \rangle$. Thus, we are required to determine the triples $(y,x,\lambda)\in \GF(q^2)\times \GF(q^2)\times\GF(q)^*$, satisfying  the system 
\begin{equation}\label{eq_12}
\left\{
\begin{array}{rcl}
(2\xi -y^{q+1})x+(2\xi+y^{q+1})x^q & = & (\alpha^q-\frac{\beta^{q+1}}{\theta})-\theta y_1^{q+1} \lambda \\[.05in]
2\theta y(x-x^q) & = & \beta+2 \xi y_1\lambda \\[.05in]
2\xi(x-x^q) & = & \theta+2\xi \theta \lambda
\end{array}
\right.,
\end{equation}
together with the condition Eq.\,\eqref{eq_10}. 
By plugging  $x=x_0+\theta x_1$, $\alpha=a_0+\theta a_1$, $x_i,a_i\in\GF(q)$, into  \eqref{eq_12} (note that $x_1\neq0$ otherwise the intersection point would coincide with $B$), we rewrite \eqref{eq_12} in the equivalent form
\begin{equation}\label{eq_13}
\left\{
\begin{array}{rcl}
4\xi  x_0 & = & a_0  \\[.05in]
2\xi   x_1 y^{q+1} & = & a_1\xi +\beta^{q+1}+\xi y_1^{q+1}\lambda \\[.05in]
4\xi x_1y & = & \beta+2 \xi y_1\lambda \\[.05in]
4\xi x_1 & = &  1+2\xi \lambda
\end{array}
\right..
\end{equation}
Hence,
\begin{equation*}
x=\frac{a_0}{4\xi}+\theta\frac{\beta-y_1}{4\xi(y-y_1)},\ \ \ \ \ \ \ \ \lambda=\frac{\beta-y}{2\xi(y-y_1)}.
\end{equation*}

By using Eq.\,\eqref{eq_10} in the second equation of \eqref{eq_13}, we come to
\[
\beta(\beta^q(y-y_1)-\beta(y^q-y_1^q) +(y^qy_1-yy_1^q))=0.
\]
Assume $\beta=0$. From  \eqref{eq_10}, it follows that $y=y_1c$, for some $c\neq 1$ with $N(c)=1$. As $x_1=-1/(4\xi(c-1))\in\GF(q)$, $c\in\GF(q)$ with $c^{q+1}=c^2=1$, that is, $c=-1$.
Assume $\beta^q(y-y_1)-\beta(y^q-y_1^q) +(y^qy_1-yy_1^q)=0$, i.e., $(\beta^q-y_1^q)(y-y_1))-(\beta-y_1)(y^q-y_1^q)=0$, where $\beta^q-y_1^q\neq 0$ (as $\Pi^\perp$ is non-singular). Then,
\begin{equation}\label{eq_14}
y=\frac{a+(\beta^q-y_1^q)y_1}{\beta^q-y_1^q}
\end{equation}
for some $a\in\GF(q)^*$. Substituting \eqref{eq_14} into \eqref{eq_10} yields  $a=2(\beta-y_1)^{q+1}$, whence $y=2\beta-y_1$. This concludes the proof.\qedhere
\end{proof}
\begin{remark}\label{rem_2}
The line  $(B^\perp\cap \Pi)^\perp$ contains precisely $q$ non-singular points, one of which is $\Pi^\perp$. Under the polarity defined by $Q^-(5,q)$, the corresponding hyperplanes share $B^\perp\cap \Pi$. For any such a hyperplane $H$, the line $\langle l,H^\perp \rangle\cap H$ coincides with the line $\sigma$ constructed from $\Pi$ \footnote{To see this, just note that these hyperplanes are the ``perp" of the non-singular points on the line $\langle B,\Pi^\perp \rangle$, and have the form $\langle (\lambda+\alpha, \beta,\theta)_2 \rangle$, for all $\lambda\in\GF(q)$. Straightforward calculations show that the corresponding line $\sigma_\lambda=\langle l,H^\perp \rangle\cap H$ coincides with $\sigma$.}. Therefore, these hyperplanes define the same involution on the generators of the quadratic cone $B^\perp\cap\Pi\cap Q^-(5,q)$.
\end{remark}
\begin{proposition}\label{lem_12}
Let $l_1,l_2,l_3$ be three distinct lines of $Q^-(5,q)$ spanning the whole space. Then, $l_1,l_2,l_3$ are in perspective if and only if, for some flag $(B,l_i)$,  the totally singular lines through $B$  concurrent with $l_j$ and $l_k$, $i\neq j\neq k\neq i$, correspond under the map $\tilde\sigma$ defined by the hyperplane containing $B$, $l_j$ and $l_k$. 
\end{proposition}
\begin{proof}
Fix $B\in l_i$. By \cite[Theorem 10.12]{taylor}, up to the Klein correspondence, we may choose coordinates such that $l_i=l=L(I,0,0)$, $l_j=m=L(0,0,I)$ and $l_k=n=L(F_0,F_1,I)$, with $B=\langle (1,0,0)_2 \rangle$. 
Let $\Pi$ be the (non-degenerate) hyperplane spanned by $B$, $m$ and $n$. Let $p_1$ and $p_2$ be the two totally singular lines on $B$ concurrent with $m$ and $n$, respectively.  Since $m^\perp=\pi(I,0,0)$, by Lemma \ref{lem_8}(i) $\Pi$ has an equation of the form
\[
\Pi: \theta (X-X^q)+\beta^{q}Y+\beta Y^q=0,
\]
for some $\beta\in\GF(q^2)^*$. 

Since the line $p_1$ has the form given by Lemma \ref{lem_8}(ii), $p_1=l_0$.  Lemma  \ref{lem_8}(ii) and (iii) imply $p_1^{\tilde\sigma}=l_{2\beta}=L(G_0,G_1,G_2)$, with 
\[
G_0(x)=(2\xi-4\beta^{q+1})x+(2\xi+4\beta^{q+1})x^q, \quad G_1(x)=4\theta\beta(x-x^q), \quad  G_2(x)=2\xi (x-x^q).
\]

By Remark \ref{rem_1}, we may assume $F_2=I$, that is, $n=L(F_0,F_1,I)$, with $F_0(x)=f_0x+g_0x^q$, $F_1(x)=f_1x+g_1x^q$. The condition that $n$ belongs to $\Pi$ is equivalent to have
\begin{equation}\label{eq_16}
\beta g_1^q+\beta^q f_1+\theta(f_0-g_0^q)=0.
\end{equation}

Therefore, $n$ is concurrent with $l_{2\beta}$ if and  only if there exist $x,\bar x\in\GF(q^2)^*$ such that
\begin{equation}\label{eq_15}
\left\{
\begin{array}{rcl}
f_0 x+g_0x^q & = & (2\xi-4\beta^{q+1})\bar x+(2\xi+4\beta^{q+1})\bar x^q  \\[.05in]
f_1 x+g_1x^q & = & 4\theta\beta(\bar x-\bar x^q) \\[.05in]
x & = &  2\xi (\bar x-\bar x^q)
\end{array}
\right..
\end{equation}
Write $\bar x=\bar x_0+\theta \bar x_1$, $\bar x_i\in\GF(q)$. Then, $x=4\xi\theta\bar x_1\neq 0$.

From the second equation of \eqref{eq_15}, we get $2\beta=\theta(f_1-g_1)$. This, together with Eq.\,\eqref{eq_16}, yields 
\begin{equation}\label{eq_17}
2f_1g_1^q+2(f_0-g_0^q)-(f_1^{q+1}+g_1^{q+1})=0.
\end{equation}
The equations \eqref{eq_5} with $f_2=1$ and $g_2=0$, applied to \eqref{eq_17}, give $f_0\in\GF(q)$, and Proposition \ref{prop_2} leads to the result.
\end{proof}
\begin{remark}\label{rem_3}
Note that if Proposition \ref{lem_12} holds for one point $B\in l_i$, then it holds for all points of $l_i$.
\end{remark}

\section{A five-class association scheme on $H(3,q^2)$}\label{sec_3}
Let $V=V(4,q^2)$  equipped with a non-degenerate Hermitian form $h:V\times V\rightarrow \GF(q^2)$. As usual,  $H(3,q^2)$ denotes the  unitary polar geometry of rank 2 defined by $h$, and it is called a {\em Hermitian surface} of $\PG(3,q^2)$. A {\em point} (resp. {\em line}) of $H(3,q^2)$ is a 1-dimensional (resp. 2-dimensional) subspace in $H(3,q^2)$, that is, totally isotropic with respect to $h$. A pair of vectors $(\x,\y)$ such that $\x$ and $\y$ are isotropic with $h(\x,\y)=1$ is called a {\em hyperbolic pair}; in this case, $\langle \x,\y \rangle$ in $\PG(3,q^2)$ is said to be a {\em hyperbolic line}. Any hyperbolic line intersects $H(3,q^2)$ in $q+1$ points.  Two distinct points $P=\langle \p \rangle$ and $Q=\langle \q \rangle$ of $H(3,q^2)$ are said to be {\em orthogonal} or {\em collinear} if $h(\p,\q)=0$; in other words, they span a totally isotropic line.

Since all non-degenerate Hermitian forms on $V$ are isometric, we may take an ordered basis $(\v_0,\v_1,\v_2,\v_3)$ for $V$ such that
\begin{equation}\label{eq_1}
h(\x,\y)=x_0y_3^q-x_1y_1^q-x_2y_2^q+x_3y_0^q,
\end{equation}
where  $\x=x_0\v_0+x_1\v_1+x_2\v_2+x_3\v_3$ and $\y=y_0\v_0+y_1\v_1+y_2\v_2+y_3\v_3$.

In \cite{shult}, Shult introduced the following function on $H(3,q^2)$.  For any  three distinct points $P=\langle \p \rangle$, $Q=\langle \q \rangle$ and $R=\langle \r \rangle$ of $H(3,q^2)$, let
\[
z(P,Q,R)=h(\p,\q)h(\q,\r)h(\r,\p)\GF(q)^*,
\]
where $\GF(q)^*$ denotes the multiplicative group of non-zero elements of $\GF(q)$. 
Then, $z(P,Q,R)$ is well-defined and 
\begin{align*}
&z(P,Q,R)=z(R,Q,P)=z(Q,P,R);\\
&z(P,Q,R)=z(Q,P,R)^q.
\end{align*}

In the multiplicative group $T=\GF(q^2)^*/\GF(q)^*\simeq Z_{(q+1)}$, with identity  $e=\GF(q)^*$, the element $t=\theta\GF(q)^*$ is the unique involution.
\begin{lemma}[\cite{shult}]\label{lem_7}
Let $P,Q,R$ be three pairwise non-collinear points of $H(3,q^2)$. Then, the span of $P,Q,R$ is a degenerate plane if and only if  $z(P,Q,R)= t$.
\end{lemma}
Let $\Gamma=T\setminus\{e,t\}$. Fix a point $P$ of $H(3,q^2)$ and consider the set $\cX$ of all the points of $H(3,q^2)$ that are not collinear with $P$.
On the set $\cX$, which consists of $q^5$  points,   we define the following relations:
\begin{enumerate}
\item[] \mbox{$R_1=\{(Q,R): z(P,Q,R)=0 \}$},
\item[] \mbox{$R_2=\{(Q,R): \langle P,Q,R \rangle\ \mathrm{is\ a\ (hyperbolic)\ line}\}$},
\item[] \mbox{$R_3=\{(Q,R): z(P,Q,R)=t\}$},
\item[] \mbox{$R_4=\{(Q,R):  z(P,Q,R)\in\Gamma\}$},
\item[] \mbox{$R_5=\{(Q,R):  z(P,Q,R)= e\}$}.
\end{enumerate}

Note that $(Q,R)\in R_1$ if and only if $Q$ is collinear with $R$ and $(Q,R)\in R_3$ if and only if $P,Q,R$ span a degenerate plane (see  Lemma \ref{lem_7}).

Set $\cR=\{R_0,R_1,\ldots,R_5\}$, where $R_0$ is the diagonal relation. We are going to prove that $\mathfrak X_P=(\cX,\cR)$ is a symmetric, hence commutative,  imprimitive association scheme.  
Clearly all the above relations are symmetric. 
We now show that all of the intersection numbers $p_{ij}^k$ are well defined.  Note that if $p_{ij}^k$ is well defined then so too is $p_{ji}^k=p_{ij}^k$. 
We will be aided by the fact that the projective unitary group $\PGU(4,q^2)$ is transitive on the set of pairs of non-collinear points of $H(3,q^2)$ \cite[Corollary 11.12]{grove}. Thus, in the computations of the parameters  we may  assume  $P=\langle (0,0,0,1) \rangle$, $Q=\langle (1,0,0,0) \rangle$, and $R=\langle (1,r_1,r_2,r_3) \rangle$, with $r_1^{q+1}+r_2^{q+1}=r_3+r_3^q$, since $R\in \cX$. Note that $z(P,Q,R)=h(\q,\r)\GF(q)^*=r_3^q\GF(q)^*$.

\begin{lemma}\label{lemma_1}
The valencies $\eta_k=p_{kk}^0$ are as follows:
$\eta_1=(q^2-1)(q+1)$, $\eta_2=q-1$, $\eta_3=(q^2-1)^2$, $\eta_4=(q^3-q)(q-1)^2$, $\eta_5=(q^3-q)(q-1)$.
\end{lemma}
\begin{proof}
We calculate $\eta_1, \eta_2, \eta_3, \eta_5$ directly, obtaining $\eta_4$ by subtraction. First,
\[
\begin{array}{ccl}
\eta_1 & = & |\{R\in\cX:(Q,R)\in R_1\}|\\
       & = & |\{R\in\cX:r_3=0\}|\\
       & = & |\{(1,r_1,r_2,0): r_1^{q+1}+r_2^{q+1}=0\}|.
 \end{array}
 \]
 
Note that $r_1, r_2 \neq 0$, otherwise $R=Q$. Fix $r_1\in\GF(q^2)^*$. There exist $q+1$ elements $r_2\in\GF(q^2)^*$ satisfying $r_2^{q+1}=-r_1^{q+1}$. Therefore, $\eta_1= (q+1)(q^2-1)$.
Next,
\[
\begin{array}{ccl}
\eta_2 & = & |\{R\in\cX:(Q,R)\in R_2\}|\\
       & = & |\{R\in\cX:R \in \langle P,Q \rangle\}|\\
       & = & |\{R\in\cX:r_1=r_2=0\}|\\
       & = & |\{(1,0,0,r_3): r_3+r_3^{q}=0\}|,
 \end{array}
 \]
 where $r_3 \neq 0$, otherwise $R=Q$. Since there exist $q$ elements $r_3\in\GF(q^2)^*$ satisfying $\Tr(r_3)=0$, we have $\eta_2= q-1$.

\[
\begin{array}{ccl}
\eta_3 & = & |\{R\in\cX:(Q,R)\in R_3\}|\\
       & = & |\{R\in\cX:r_3\in\ \theta \GF(q)^*\}|\\
       & = & |\{(1,r_1,r_2,\theta a):a\in\GF(q)^*, r_1^{q+1}+r_2^{q+1}=0\}|.
 \end{array}
 \]

Note that $r_1, r_2 \neq 0$, otherwise $(Q,R)\in R_2$. For fixed $r_1\in\GF(q^2)^*$, there are $q+1$ elements $r_2\in \GF(q^2)^*$ such that $r_2^{q+1}=-r_1^{q+1}$. Therefore, as $r_3=\theta a, a \in\GF(q)^*$, $\eta_3=(q+1)(q^2-1)(q-1)$. Finally,
\[
\begin{array}{ccl}
\eta_5 & = & |\{R\in\cX:(Q,R)\in R_5\}|\\
       & = & |\{R\in\cX:r_3\in\GF(q)^*\}|\\
       & = & |\{(1,r_1,r_2,r_3):r_3\in\GF(q)^*, r_1^{q+1}+r_2^{q+1}=2r_3\}|.
 \end{array}
 \]

Fix $r_3\in\GF(q)^*$. Then, for any $r_1\in \GF(q^2)$ such that $r_1^{q+1}\neq 2r_3$, we find $q+1$ non-zero elements $r_2\in\GF(q^2)$ which satisfy $r_2^{q+1}=2r_3-r_1^{q+1}$; for any $r_1\in \GF(q^2)$ with $r_1^{q+1}= 2r_3$,  $r_2=0$ necessarily. Therefore, $\eta_5= ((q+1)(q^2-q-1)+q+1)(q-1)$.
 
Finally, $\eta_4=|\cX|-(1+\eta_1+ \eta_2+\eta_3+\eta_5)=q(q^2-1)(q-1)^2$.
\end{proof}

\begin{lemma}\label{lem_2}
The intersection numbers $p_{1j}^k$ are well defined. They are collected in the following intersection matrix $L_1$ whose $(k,j)-$entry is $p_{1j}^k$:
\[
L_1=\begin{pmatrix}
0 & (q^2-1)(q+1) & 0 & 0            & 0           & 0 \\
1 & q^2-2        & 0 & q(q-1)       & q(q-1)^2    & q(q-1)\\
0 &   0          & 0 & (q-1)(q+1)^2 &  0          & 0  \\
0 &   q          & 1 & 2(q^2-q-1)   & q(q-1)^2    & q(q-1)\\
0 & q+1          & 0 & q^2-1        & q^3-q^2-2q   & q^2-1 \\
0 & q+1          & 0 & q^2-1        & (q+1)(q-1)^2 & (q-2)(q+1)
 \end{pmatrix}
\]
\end{lemma}

\begin{proof}
 To check that  $p_{1j}^k$ is well defined, for any pair $(X,Q)\in R_k$ we count the number of points $R$ collinear with $Q$ and  $j-$related with $X$. As $R=\langle (1,r_1,r_2,r_3) \rangle$ is   collinear with $Q=\langle (1,0,0,0) \rangle$, we have $r_3=0$, so $r_1^{q+1}+r_2^{q+1}=0$. 

Assume $k=1$, and let $X$ be collinear with $Q$. Then, $X=\langle (1,x_1,x_2,0) \rangle$ and   
\[
z(P,R,X)=(r_1x_1^q+r_2x_2^q)\GF(q)^*.
\]

Any pair $(r_1,r_2)\in \GF(q^2)^2\setminus\{(0,0), (x_1,x_2)\}$ such that $r_1x_1^q+r_2x_2^q=0$ is of type $(-(x_2/x_1)^qa,a)$, for some $a\in\GF(q^2)^*\setminus\{x_2\}$. This implies  $p_{11}^1=q^2-2$. When $R\in\langle P,Q \rangle$ it is easy to check that $p_{12}^1=0$.
 
 Let $(r_1,r_2)\in \GF(q^2)^2\setminus\{(0,0), (x_1,x_2)\}$ such that $r_1x_1^q+r_2x_2^q\in\theta\GF(q)^*$. Then, $\theta(r_1x_1^q+r_2x_2^q)\in\theta^2\GF(q)^*=\GF(q)^*$.  Since $(\theta x_1)^{q+1}+(\theta x_2)^{q+1}=\theta^{q+1}(x_1^{q+1}+x_2^{q+1})=0$, counting the elements 3-related with $X=\langle (1,x_1,x_2,0) \rangle$ is equivalent to counting the elements 5-related with $\langle (1,\theta x_1,\theta x_2,0) \rangle$, that is, $p_{13}^1=p_{15}^1$. So we assume $r_1x_1^q+r_2x_2^q\in \GF(q)^*$. For any fixed  $a\in\GF(q)^*$, we have  $r_2=(a-r_1x_1^q)/x_2^q$. From $r_1^{q+1}+r_2^{q+1}=0$, it follows that $\Tr(r_1x_1^q)=a$. As for any given $a$ there are $q$ elements $x\in\GF(q^2)$ such that $\Tr(x)=a$,  we see that $p_{13}^1=p_{15}^1=q(q-1)$.
 Finally $p_{14}^1=\eta_1-(p_{10}^1+p_{11}^1+p_{12}^1+p_{13}^1+p_{15}^1)=q(q-1)^2$.
 
 Assume $k=2$, and let  $X\in\langle P,Q \rangle\cap H(3,q^2)$. Then, $X=\langle (1,0,0,\theta a) \rangle$, for some $a\in\GF(q)^*$, and $z(P,R,X)=\theta^q\GF(q)^*=t$. 
This implies $p_{11}^2=p_{12}^2=p_{14}^2=p_{15}^2=0$ and $p_{13}^2=\eta_1=(q-1)(q+1)^2$.

 Assume $k=3$, and let  $X\in\cX$ such that $\langle P,Q,X \rangle$ is a degenerate plane. Then, $X=\langle (1,x_1,x_2,\theta a) \rangle$ for some $a\in\GF(q)^*$, with $x_1^{q+1}+x_2^{q+1}=0$, $x_1\neq0\neq x_2$. We have  $z(P,R,X)=(\theta a+r_1x_1^q+r_2x_2^q)\GF(q)^*$.

It is easy to see that $p_{11}^3=p_{13}^1=q(q-1)$ and $p_{12}^3=1$.  

Let $(r_1,r_2)\in \GF(q^2)^2\setminus\{(0,0),(x_1,x_2)\}$ such that $r_1x_1^q+r_2x_2^q=\theta(b-a)$ for some $b\in\GF(q)^*$. For any given $b \neq a$, we have $q$ such pairs,  and this gives $q(q-2)$ pairs as $b$ varies in $\GF(q)^*\setminus\{a\}$. Let $b=a$. Then, the  number of pairs $(r_1,r_2)\in \GF(q^2)^2\setminus\{(0,0),(x_1,x_2)\}$ with $r_1^{q+1}+r_2^{q+1}=0$ and $r_1x_1^q+r_2x_2^q=0$ is $q^2-2=p^1_{11}$. Therefore,  $p_{13}^3=2(q^2-q-1)$.

Let $(r_1,r_2)\in \GF(q^2)^2\setminus\{(0,0),(x_1,x_2)\}$ such that $r_1x_1^q+r_2x_2^q=b-\theta a$, for some $b\in\GF(q)^*$. For any fixed  $b\in\GF(q)^*$, we have  $r_2=(b-\theta a- r_1x_1^q)/x_2^q$. By plugging this into  $r_1^{q+1}+r_2^{q+1}=0$, we get  $\Tr((b+\theta a)r_1x_1^q)=b^2-\xi a^2=\N(b-\theta a)$. As there are $q$ elements $x\in\GF(q^2)$ such that $\Tr(x)=c$, for any given $c\in\GF(q)$, we get $p_{15}^3=q(q-1)$. 
Finally $p_{14}^3=\eta_1-(p_{10}^3+p_{11}^3+p_{12}^3+p_{13}^3+p_{15}^3)=q(q-1)^2$.
 
Assume $k=4$, and let  $X\in\cX$ such that $z(P,Q,X)\in\Gamma$. Then, $X=\langle (1,x_1,x_2,w^i a) \rangle$ for some  $i\neq 0,(q+1)/2$ and $a\in\GF(q)^*$, with $x_1^{q+1}+x_2^{q+1}=a\Tr(w^i)$. We have  
 \[
z(P,R,X)=(w^{iq}a-r_1x_1^q-r_2x_2^q)\GF(q)^*.
\]

Let $w^{iq}a-r_1x_1^q-r_2x_2^q=0$. Assume $x_2=0$. Then, $x_1\neq 0$ (otherwise $(X,Q)\in R_2$), $r_1=w^{iq}a/x_1^q$ and $r_2^{q+1}=-\N(r_1)=-a\N(w^i)/\Tr(w^i)$. Therefore, in this case there are $q+1$ pairs $(r_1,r_2)$ which satisfy the above properties. 

Assume $x_2\neq0$. Then, $r_2=(w^{iq}a-r_1x_1^q)/x_2^q$. By plugging this into $r_1^{q+1}+r_2^{q+1}=0$, we get 
\[
\Tr(w^i)r_1^{q+1}-\Tr(w^ir_1x_1^q)+a\N(w^i)=0, 
\]
or
\[
(r_1,1)\begin{pmatrix}
\Tr(w^i) & -w^{i}x_1^q\\
-w^{iq}x_1 & a\N(w^i)
\end{pmatrix}
\begin{pmatrix}
r_1^q\\
1
\end{pmatrix}=0,
\]
with $\N(w^i)(a\Tr(w^i)-x_1^{q+1})\neq0$, as $x_2\neq0$. Since $\Tr(w^i)\neq0$, the  above non-singular Hermitian matrix defines a unitary form of $V(2,q^2)$ not admitting the point $\langle (1,0) \rangle$ as a totally isotropic point. Therefore, there are $q+1$ values for $r_1$, whence $q+1$ pairs $(r_1,r_2)$ satisfying the above properties. Hence, $p_{11}^4=q+1$.

Let $R\in\langle P,X \rangle\cap H(3,q^2)$, i.e., $R=\langle (1,r_1,r_2,0) \rangle$. Since $r_1^{q+1}+r_2^{q+1}=a\Tr(w^i)\neq 0$, we get $R\notin H(3,q^2)$, whence  $p_{12}^4=0$. By arguing as we did for $p_{11}^4$, we find $p_{13}^4=p_{15}^4=q^2-1$.

Finally, $p_{14}^4=\eta_1-(p_{10}^4+p_{11}^4+p_{12}^4+p_{13}^4+p_{15}^4)=q^3-q^2-2q$.
 
 Assume $k=5$, and let  $X\in\cX$ such that $z(P,Q,X)=e$. Then, $X=\langle (1,x_1,x_2,a) \rangle$ for some  $a\in\GF(q)^*$, with $x_1^{q+1}+x_2^{q+1}=2a$. We have  
 \[
z(P,R,X)=(a-r_1x_1^q-r_2x_2^q)\GF(q)^*.
\]

We may argue as above  to show that $p_{11}^5=q+1$,  $p_{12}^5=0$, $p_{13}^5=q^2-1$,  $p_{15}^5=(q-2)(q+1)$, and $p_{14}^5=(q+1)(q-1)^2$.
\end{proof}

\begin{lemma}\label{lem_3}
The intersection numbers $p_{2j}^k$ are well defined. They are collected in the following intersection matrix $L_2$ whose $(k,j)-$entry is $p_{2j}^k$:
\[
L_2=\begin{pmatrix}
0 & 0 & q-1 & 0            & 0           & 0 \\
0 & 0 & 0   & q-1       & 0   & 0 \\
1 & 0 & q-2 & 0 &  0          & 0  \\
0 & 1 & 0 & q-2 & 0    & 0\\
0 & 0 & 0 & 0  & q-2   & 1 \\
0 & 0 & 0 & 0  & q-1 & 0
 \end{pmatrix}
\]
\end{lemma}
\begin{proof}
 To check that  $p_{2j}^k$ is well defined, for any pair $(X,Q)\in R_k$ we count the number of points $R\in\cX$ such that $R$ is on the hyperbolic line spanned by $P$ and $Q$  with  $(R,X)\in R_j$.
It is easily seen that $R=\langle (1,0,0,a\theta) \rangle$, for some $a\in\GF(q)^*$. 
We have $p_{21}^k=p_{12}^k$, for $k=0,\ldots,5$.

Assume $k=1$. Then, $X=\langle (1,x_1,x_2,0) \rangle$, with $x_1\neq0\neq x_2$. Therefore, there is  no point $R$  on the hyperbolic line spanned by $P$ and $X$ giving $p_{22}^1=0$.
In addition, 
\[
z(P,R,X)=\theta a\GF(q)^*=t.
\]
 
 This implies $p_{24}^1=p_{25}^1=0$ and $p_{23}^1=q-1$.
 
 Assume $k=2$, and let   $X=\langle (1,0,0,b\theta ) \rangle$, for some $b\in\GF(q)^*$. Since $R\in\langle P,Q \rangle\setminus\{Q,X\}$, we have $p_{22}^2=q-2$. This also implies $p_{23}^2=p_{24}^2=p_{25}^2=0$.

 Assume $k=3$, and let  $X=\langle (1,x_1,x_2,b\theta ) \rangle$ for some $b\in\GF(q)^*$, with $x_1^{q+1}+x_2^{q+1}=0$, $x_1\neq0\neq x_2$. Then, $p_{22}^3=0$.  
In addition,
\[
z(P,R,X)=\theta(a+b)\GF(q)^*.
\]
For $a=-b$, $(R,X)\in R_1$. For $a\neq -b$, $z(P,R,X)=t$. Therefore, $p_{24}^3=p_{25}^3=0$ and $p_{23}^3=q-2$.  
 
  Assume $k=4$, and let  $X=\langle (1,x_1,x_2,x_3) \rangle$ for some  $x_3\notin\GF(q)^*\cup \theta\GF(q)^*$. Then, $p_{22}^4=0$, otherwise $x_3\in \theta\GF(q)^*$.
  In addition, 
\[
z(P,R,X)=(a\theta+x_3^q)\GF(q)^*.
\]

Therefore, $p_{23}^4=0$. To calculate $p_{25}^4$ we see that $a\theta+x_3^q\in\GF(q)^*$ if and only if $a=(x_3-x_3^q)/2\theta$. Therefore, $p_{25}^4=1$ and $p_{24}^4=q-2$.   

 Assume $k=5$, and let  $X=\langle (1,x_1,x_2,b) \rangle$ for some  $b\in\GF(q)^*$. Then, $p_{22}^5=0$, otherwise $b=0$. In addition,
 \[
z(P,R,X)=(a\theta+b)\GF(q)^*.
\]

Therefore, there are no $a\in \GF(q)^*$ such that $a\theta+b\in\GF(q)^*$ or $a\theta+b\in\theta\GF(q)^*$. This implies, $p_{25}^5=p_{23}^5=0$ and $p_{24}^5=q-1$. 
\end{proof}
\begin{lemma}\label{lem_4}
The intersection numbers $p_{3j}^k$ are well defined. They are collected in the following intersection matrix $L_3$ whose $(k,j)-$entry is $p_{3j}^k$:
{\footnotesize
\[
L_3=\begin{pmatrix}
0 & 0 & 0 & (q^2-1)^2           & 0           & 0 \\
0 & q(q-1) & q-1   & 2(q-1)(q^2-q-1)       & q(q-1)^3   &q(q-1)^2 \\
0 & (q-1)(q+1)^2 & 0 & (q^2-1)(q^2-q-2) &  0          & 0  \\
1 & 2(q^2-q-1) & q-2 & 2q^3-5q^2+q+4 & q(q-1)^3    & q(q-1)^2\\
0 & q^2-1 & 0 &  (q^2-1)(q-1) & q^4-2q^3-q^2+3q+1   & q(q+1)(q-2) \\
0 & q^2-1 & 0 & (q^2-1)(q-1)  & q(q^2-1)(q-2) & (q^2-1)(q-1)
 \end{pmatrix}
\]}
\end{lemma}
\begin{proof}
To check that  $p_{3j}^k$ is well defined, for any pair $(X,Q)\in R_k$ we count the number of points $R$ which are $3-$related  with $Q$ and  $j-$related with $X$. It is easily seen that $R=\langle (1,r_1,r_2,a\theta) \rangle$, for some $a\in\GF(q)^*$, with $r_1^{q+1}+r_2^{q+1}=0, r_1 \neq 0 \neq r_2$ (otherwise $(Q,R)\in R_2$).
From the previous calculations, we already have $p_{31}^k=p_{13}^k$, $p_{32}^k=p_{23}^k$, for $k=0,\ldots,5$.

Assume $k=1$. Then, $X=\langle (1,x_1,x_2,0) \rangle$, with $x_1^{q+1}+x_2^{q+1}=0$, $x_1\neq0\neq x_2$. Therefore, $(R, X)\in R_3$ if and only if there exists $b\in\GF(q)^*$ such that $r_1x_1^q+r_2x_2^q=\theta (a-b)$. First of all, suppose $a\neq b$. As $\theta (r_1x_1^q+r_2x_2^q) \in \GF(q)^*$ and $(\theta x_1)^{q+1}+(\theta x_2)^{q+1}= \theta ^{q+1}(x_1^{q+1}+x_2^{q+1})=0$, we may consider $(r_1, r_2) \in\GF(q^2)^* \times \GF(q^2)^* $ such that $r_1x_1^q+r_2x_2^q \in \GF(q)^*$ (see the calculation of $p^1_{15}$). Let $c\in \GF(q)^*$, and write $r_2=x_2^{-q}(c-r_1 x^q_1)$. By using $r_1^{q+1}+r_2^{q+1}=0$, it follows that $\Tr(r_1 x_1^q)=c$, which is true for exactly $q$ elements $r_1 \in \GF(q)^*$. Therefore, for $b \in \GF(q)^*\setminus\{a\}$, $q(q-1)(q-2)$ is the number of the triples $(r_1, r_2, a)$, $c$ being one-to-one with $b$. Now, consider $a=b$. Then, $r_1x_1^q+r_2x_2^q=0$, for $(r_1,r_2) \in \GF(q^2)^2\setminus\{(0,0), (x_1,x_2)\}$. So, in the case $a=b$, the number of the triples $(r_1, r_2, a)$ is equal to $(q-1)(q^2-2)$. Finally, by summing the previous two quantities, we have $p^1_{3 3}=2(q-1)(q^2-q-1)$.

We show now that $p^1_{3 5}$ is well-defined, by explicitly calculating it. Thus, $(R, X)\in R_5$ if and only if there exists $b\in\GF(q)^*$ such that $r_1x_1^q+r_2x_2^q=\theta a-b$. By deriving $r_2$ from the previous expression and considering $r_1^{q+1}+r_2^{q+1}=0$ as usual, we obtain $\Tr((\theta a+b)(r_1 x_1^q))=\theta^2a^2-b^2$, which is satisfied by exactly $q$ elements $r_1$. As $a,b\in \GF(q)^*$, $p^1_{3 5}=q(q-1)^2$. 

Finally, $p^1_{3 4}=\eta_3-(p^1_{3 0}+p^1_{3 1}+p^1_{3 2}+p^1_{3 3}+p^1_{3 5})=q(q-1)^3$.

Assume $k=2$. Then, $X=\langle (1,0,0,\theta b) \rangle$ with $b\in \GF(q)^*$, and 
\begin{equation*}
z(P,R,X)=\theta (a+b)\GF(q)^*.
\end{equation*}
This means that, for $a+b\neq 0$ (otherwise $p^2_{3 1}=p^2_{1 3}$), $(R,X)\in R_3$, from which $p^2_{3 j}=0$, for $ j=0, 2, 4, 5$, and $p^2_{3 3}=\eta_3-p^2_{3 1}=(q^2-1)(q^2-q-2)$.

Take $k=3$. Then, $X=\langle (1,x_1,x_2,\theta b) \rangle$, with $x_1^{q+1}+x_2^{q+1}=0$, $x_1\neq0\neq x_2$, for some $b \in\GF(q)^*$. Here, $(R,X)\in R_3$ if and only if  $r_1x_1^q+r_2x_2^q=\theta (c-b+a)$ for some $c \in \GF(q)^*$. We need, at this point, to distinguish different cases. First of all, let $c \neq b$ and $a \neq b-c$. Once multiplied the right hand-side by $\theta$, in order to simplify the calculation, we may equivalentely look at the triples $(r_1,r_2,a)$ such that $r_1x_1^q+r_2x_2^q=\theta ^2 (c-b+a)$, i.e., $r_2=x_2^{-q}(\theta ^2 (c-b+a)-r_1x_1^q)$. By using $r_1^{q+1}+r_2^{q+1}=0$, we find $\Tr(r_1x_1^q)=\theta ^2 (c-b+a)$, which provides $q$ values for $r_1$. Hence, for $c \in \GF(q)^* \setminus\{ b\}, a \in \GF(q)^* \setminus \{b-c\}$, the number of the triples $(r_1,r_2,a)$ is $q(q-2)^2$. Consider now the subcase $c-b+a=0$. Then, $r_2=x_2^{-q}r_1x_1^ q$ for $r_1 \in \GF(q^2)^*\setminus\{x_1\}$ (otherwise $r_2=x_2$ and $(R, X)\in R_2$). So, for $a=b-c\neq 0$, the number of the triples $(r_1, r_2, a)$ is equal to $(q-2)(q^2-2)$. Finally, we explore the case $c=b$. As before, the computation may be reduced to considering the triples $(r_1,r_2,a)$ such that $r_1x_1^q+r_2x_2^q=\theta ^2 a$, i.e., $r_2=x_2^{-q}(\theta ^2 a-r_1x_1^q)$. Since $r_1^{q+1}+r_2^{q+1}=0$ yields $\Tr(r_1x_1^q)=\theta^2 a$, the number of all possible choices for $r_1$ is $q$. To conclude, by putting together all the previous results, we have $p_{3 3}^3=q(q-2)^2+(q-2)(q^2-2)+q(q-1)=2 q^3-5q^2+q+4$.

We count now $R=\langle (1,r_1,r_2,a\theta) \rangle$ such that $(R,X)\in R_5$. This condition means $r_1x_1^q+r_2x_2^q=c+\theta (a-b)$, for some $c\in \GF(q)^*$. By plugging this into $r_1^{q+1}+r_2^{q+1}=0$,  we get  $\Tr((c-\theta(a-b))r_1x_1^q)=c^2-\theta ^2 (a-b)^2$. Thus, there are $q$ elements $r_1$ satisfying this equation, and $p^3_{3 5}=q(q-1)^2$, as $c, a \in \GF(q)^*$.

Finally, in order to conclude the study of the case $k=3$, write $p^3_{3 4}=\eta_3-(p^3_{3 0}+p^3_{3 1}+p^3_{3 2}+p^3_{3 3}+p^3_{3 5})=q(q-1)^3$.

Take $k=4$. Therefore, $X=\langle (1,x_1,x_2,\omega^i b) \rangle$, where $x_1^{q+1}+x_2^{q+1}=\Tr(\omega^i b)$, $b\in \GF(q)^*$. Suppose $R$ is $3-$related with $X$. This means there exists $c\in \GF(q)^*$ such that $r_1x_1^q+r_2x_2^q=\omega ^{i q} b+\theta (a-c)$. At this point, we distinguish when $x_2$ is zero and when it is not. Suppose $x_2=0$. As $x_1 \neq 0$, we have $r_1=x_1^{-q}(\omega ^{i q} b+\theta (a-c))$, hence $r_2^{q+1}=-r_1^{q+1}=-x_1^{-(q+1)}\N(\omega ^{i q} b+\theta (a-c))$, which is satisfied by exactly $q+1$ elements $r_2$. Since $a , c \in \GF(q)^*$, the number of the triples $(r_1,r_2,a)$ is $(q+1)(q-1)^2$. Now suppose $x_2 \neq 0$. Then, $r_2=(\theta (a-c) + \omega^{iq}b-r_1x_1^q)/x_2^q$. By plugging this into $r_1^{q+1}+r_2^{q+1}=0$,  we get 
\[
r_1^{q+1}\Tr(\omega^i b)-\Tr((\omega^{iq} b +\theta(a-c))x_1 r_1^q)+\N(w^{iq}b+\theta(a-c))=0, 
\]
or
\[
(r_1,1)\begin{pmatrix}
\Tr(\omega^i b) & -(\omega^{i}b-\theta (a-c))x_1^q\\
-(\omega^{i q}b+\theta (a-c))x_1 & \N(\omega^{iq}b+\theta(a-c))
\end{pmatrix}
\begin{pmatrix}
r_1^q\\
1
\end{pmatrix}=0,
\]
with $\N(\omega^{iq}b+\theta(a-c))(\Tr(\omega^i b)-x_1^{q+1})\neq0$, as $x_2\neq0$. Since $\Tr(\omega^i)\neq0$, the  above non-singular Hermitian matrix defines a unitary form of $V(2,q^2)$ not admitting the point $\langle (1,0) \rangle$ as a totally isotropic point. Therefore, there are $q+1$ values for $r_1$. As $a , c \in \GF(q)^*$, the number of the triples $(r_1,r_2,a)$ is again $(q+1)(q-1)^2$. Hence, we may write $p^4_{3 3}=(q+1)(q-1)^2$.

Suppose now $(R, X)\in R_5$. This is equivalent to having, for some $c\in \GF(q)^*$, $r_1x_1^q+r_2x_2^q=\omega ^{i q} b+\theta a+c$. The second member is zero if and only if the pair $(a, c)$ concides with the unique pair $(a', c')$ such that $-\omega ^{i q} b=\theta a'+c'$, $\{1,\theta\}$ being a basis for $\GF(q^2)$ over $\GF(q)$. Anyway, for this pair of values, we would have $r_1x_1^q+r_2x_2^q=0$, from which $r_1=0$ (or $r_2=0$), a contradiction. Then, let $(a,c)\neq(a',c')$. At this point, by proceeding as before for the computation of $p^4_{3 3}$, we distinguish the case $x_2=0$ from $x_2\neq0$, and $p^4_{3 5}=(q+1)((q-1)^2-1)=(q+1)q(q-2)$ is obtained. 

Finally,  $p^4_{3 4}=\eta_3-(p^4_{3 0}+p^4_{3 1}+p^4_{3 2}+p^4_{3 3}+p^4_{3 5})=q^4-2q^3-q^2+3q+1$. 

Assume $k=5$, and let  $X=\langle (1,x_1,x_2,b) \rangle$ for some  $b\in\GF(q)^*$, with $x_1^{q+1}+x_2^{q+1}=2b$. Then, $(R, X) \in R_3$ if and only if there exists $c\in \GF(q)^*$ such that $r_1x_1^q+r_2x_2^q=\theta(c-a)-b$. First of all, take $x_2=0$. Here, $r_1=x_1^{-q}(\theta(c-a)-b)$ ($x_1\neq0$) and $r_2^{q+1}=-r_1^{q+1}=-x_1^{-(q+1)}\N(\theta(c-a)-b)$. Since for any $a,c \in \GF(q)^*$ there are $q+1$  values of $r_2$ solving the previous equation, we have $(q+1)(q-1)^2$ triples $(r_1,r_2,a)$ when $x_2=0$. Now, take $x_2\neq0$. Then, $r_2=(\theta (c-a)- b-r_1x_1^q)/x_2^q$. By plugging this into $r_1^{q+1}+r_2^{q+1}=0$,  we obtain 
\[
r_1^{q+1}2 b-\Tr((\theta(c-a)- b)x_1 r_1^q)+\N(\theta(c-a)- b)=0, 
\]
or
\[
(r_1,1)\begin{pmatrix}
2 b & -(-\theta(c-a)- b)x_1^q\\
-(\theta(c-a)- b)x_1 & \N(\theta(c-a)- b)
\end{pmatrix}
\begin{pmatrix}
r_1^q\\
1
\end{pmatrix}=0,
\]
with $\N(\theta(c-a)- b)(2 b-x_1^{q+1})\neq0$, as $x_2\neq0$. Since $b \neq 0$, the  above non-singular Hermitian matrix defines a unitary form of $V(2,q^2)$ not admitting the point $\langle (1,0) \rangle$ as a totally isotropic point. Therefore, there are $q+1$ values for $r_1$. As $a , c \in \GF(q)^*$, the number of the triples $(r_1,r_2,a)$ is again $(q+1)(q-1)^2$. Thus, we may write $p^5_{3 3}=(q+1)(q-1)^2$.

In order to complete the entries of $L_3$, the case $(R,X)\in R_5$ ($k=5$) remains to be studied, as $p^5_{3 4}$ will be obtained by taking a difference. Considering $(R,X)\in R_5$ is equivalent to having $r_1x_1^q+r_2x_2^q=c-b-\theta a$, for some $c\in \GF(q)^*$. Since $x_1^{q+1}+x_2^{q+1}=2b$, with $b\in \GF(q)^*$, we may distinguish the two cases $x_2=0$ and $x_2 \neq 0$, and then proceed exactly as before in the computation of  $p^5_{3 3}$, so getting $p^5_{3 5}=p^5_{3 3}=(q+1)(q-1)^2$. 

Finally, $p^5_{3 4}=\eta_3-(p^5_{3 0}+p^5_{3 1}+p^5_{3 2}+p^5_{3 3}+p^5_{3 5})=(q^2-1)q(q-2)$.
\end{proof}

\begin{lemma}\label{lem_5}
The intersection numbers $p_{5j}^k$ are well defined. They are collected in the following intersection matrix $L_5$ whose $(k,j)-$entry is $p_{5j}^k$:
\[
L_5=\begin{pmatrix}
0 & 0 & 0 & 0        & 0           & (q^3-q)(q-1) \\
0 & q(q-1) & 0   & q(q-1)^2       & q(q-1)^3   &q(q-2)(q-1) \\
0 & 0 & 0 & 0 &  (q^3-q)(q-1)        & 0  \\
0 & q(q-1) & 0 & q(q-1)^2 & q^2(q-1)(q-2)    & q(q-1)^2\\
0 & q^2-1 & 1 &  q(q+1)(q-2) & q^4-3q^3+q^2+4q-1   & q^3-2q^2-q+1 \\
1 & (q-2)(q+1) & 0 & (q^2-1)(q-1)  & q^4-3q^3+q^2+2q-1 & q^3-2q^2+q+1 
 \end{pmatrix}
\]
\end{lemma}

\begin{proof}
To check that  $p_{5j}^k$ is well defined, for any pair $(X,Q)\in R_k$ we count the number of points $R$ that are $5-$related  with $Q$ and  $j-$related with $X$. It is easy to see that $R=\langle (1,r_1,r_2,a) \rangle$, for some $a\in\GF(q)^*$, with $r_1^{q+1}+r_2^{q+1}=2 a$.
From the previous calculations, we already have $p_{51}^k=p_{15}^k$, $p_{52}^k=p_{25}^k$, $p_{53}^k=p_{35}^k$, for $k=0,\ldots,5$.

Let $k=1$. Thus, $X=\langle (1,x_1,x_2,0) \rangle$, with $x_1^{q+1}+x_2^{q+1}=0$, $x_1\neq0\neq x_2$. Therefore, $(R, X)\in R_5$ if and only if there exists $b\in\GF(q)^*$ such that $r_1x_1^q+r_2x_2^q=a-b$. Note that for $a=b$, there is no triple $(r_1,r_2,a)$ satysfing the previous equation because of the conditions on $(x_1, x_2)$. For this reason, set $a\neq b$. By writing $r_2=x_2^{-q}(a-b-r_1x_1^q)$ and substituting it into $r_1^{q+1}+r_2^{q+1}=2 a$, we get $\Tr(r_1x_1^q)=(a-b)^{-1}((a-b)^2-2 a x_2^{q+1})$, that is satisfied by $q$ elements $r_1$. Therefore, as $b\in \GF(q)^*, a \in \GF(q)^*\setminus\{b\}$, we obtain $p^1_{5 5}=(q-2)(q-1)q$.

Then, $p^1_{5 4}=\eta_5-(p^1_{5 0}+p^1_{5 1}+p^1_{5 2}+p^1_{5 3}+p^1_{5 5})=q(q-1)^3$.

Assume $k=2$, and let $X=\langle (1,0,0,\theta b) \rangle$ with $b\in \GF(q)^*$. Note that  
\begin{equation*}
z(P,R,X)=(a-\theta b)\GF(q)^* \in \Gamma, 
\end{equation*}
i.e., $j=4$, from which $p^2_{5 5}=0$, $p^2_{5 4}=\eta_5$.

Assume $k=3$, and let $X=\langle (1,x_1,x_2,\theta b) \rangle$, with $x_1^{q+1}+x_2^{q+1}=0$, $x_1\neq0\neq x_2$. Thus, $(R, X)\in R_5$ if and only if there exists $c\in \GF(q)^*$ such that $r_1x_1^q+r_2x_2^q=c+a-\theta b$, i.e., $r_2=x_2^{-q}(c+a-\theta b-r_1x_1^q)$. By using $r_1^{q+1}+r_2^{q+1}=2 a$, it follows that $\Tr((c+a+\theta b)r_1x_1^q)=(c+a)^2-\theta^2 b^2-2a$. Since the latter equation in the unknown $r_1$ admits $q$ solutions, for any $a,c \in \GF(q)^*$, we have $q(q-1)^2$ triples $(r_1,r_2,a)$, i.e., $p^3_{5 5}=q(q-1)^2$.

Therefore, $p^3_{5 4}=\eta_5-(p^3_{5 0}+p^3_{5 1}+p^3_{5 2}+p^3_{5 3}+p^3_{5 5})=q^2(q-1)(q-2)$.

Take $k=4$. Thus, $X=\langle (1,x_1,x_2,\omega^i b) \rangle$, where $x_1^{q+1}+x_2^{q+1}=\Tr(\omega^i b)$, $b\in \GF(q)^*$. Suppose $R$ is $5-$related with $X$. This means there exists $c\in \GF(q)^*$ such that $r_1x_1^q+r_2x_2^q=\omega ^{i q} b+a-c$. At this point, we study separately when $x_2$ is zero and when it is not. Suppose $x_2=0$. As $x_1 \neq 0$, we find $r_1=x_1^{-q}(\omega ^{i q} b+a-c)$, hence $r_2^{q+1}=-r_1^{q+1}+2 a=-x_1^{-(q+1)}\N(\omega ^{i q} b+a-c)+2 a$. The latter equation is satisfied by exactly $q+1$ elements $r_2$ if $-r_1^{q+1}+2 a\neq0$, otherwise $r_2=0$. So it is necessary to study the case $-r_1^{q+1}+2 a=0$. Consider $r_1^{q+1}=2 a$, i.e., $x_1^{-(q+1)}\N(\omega ^{i q} b+a-c)=2 a$. By writing this expression explicitly, we find that the elements $a\in\GF(q)^*$ for which $r_1^{q+1}=2 a$ are the solutions of the equation 
\begin{equation}\label{eq_2}
X^2+(\Tr(\omega^i b)-2x_1^{q+1}-2c)X+c^2-c\Tr(\omega^i b)+\N(\omega^i b)=0,
\end{equation}
whose discriminant is
\begin{align*}
\Delta&=8 c x_1^{q+1}-4\N(\omega^i b)+(\Tr(\omega^i b)-2x_1^{q+1})^2 \\ & =  8 c x_1^{q+1}+b^2(\omega^{iq}-\omega^{i})^2.
\end{align*}

For $\overline{c}=-b^2(\omega^{iq}-\omega^{i})^2/8 x_1^{q+1}$, $\Delta=0$ holds, i.e., there is a unique $X=\overline{a}$ satisfying \eqref{eq_2}, to which $\overline r_1=x_1^{-q}(\omega ^{i q} b+\overline{a}-\overline{c})$ and $\overline r_2=0$ correspond. Therefore, for $c=\overline{c}$, we get the triple $(\overline{r_1},0,\overline{a})$ and further $(q+1)(q-2)$ triples $(r_1,r_2,a)$ with $r_1=x_1^{-q}(\omega ^{i q} b+a-\overline{c})$, $a\neq \overline{a}$, $r_2\neq0$.

Assume $\Delta$ is a square in $\GF(q)^*$, i.e., $\Delta$ is a $\frac{q-1}{2}-$th root of the unity. Precisely one element $c$ corresponds with any such a root, and for any such a $c$, there are two values of $X$ satisfying (\ref{eq_2}).
Therefore, for every fixed $c$ among the previous ones, we get two triples of type $(r_1,0,a)$ and further $(q+1)(q-3)$ triples with $r_2\neq0$, i.e., in total we have $\frac{q-1}{2}(2+(q+1)(q-3))$.

For the remaining $q-1-(1+\frac{q-1}{2})= \frac{q-3}{2}$ values of $c\in \GF(q)^*$, $\Delta$ is a non-square, i.e., there is no $a\in \GF(q)^*$ making  $-r_1^{q+1}+2 a$ equal to zero. This means that, here, in the light of the initial considerations on both $r_1$ and $r_2$, we have $\frac{q-3}{2}(q-1)(q+1)$.
To  sum up, for $x_2=0$, the number we seek is $1+(q-2)(q+1)+\frac{q-1}{2}(2+(q+1)(q-3))+\frac{q-3}{2}(q-1)(q+1)=q^3-2q^2-q+1$.

Now suppose $x_2 \neq 0$. Then, $r_2=(a-c+ \omega^{iq}b-r_1x_1^q)/x_2^q$. By plugging this into $r_1^{q+1}+r_2^{q+1}=2a$,  we get 
\[
r_1^{q+1}\Tr(\omega^i b)-\Tr((\omega^{i} b +a-c)r_1 x_1^q)+\N(w^{i}b+a-c)-2ax_2^{q+1}=0, 
\]
or
\[
(r_1,1)\begin{pmatrix}
\Tr(\omega^i b) & -(\omega^{i}b+a-c)x_1^q\\
-(\omega^{i q}b+a-c)x_1 & \N(\omega^{i}b+a-c)-2ax_2^{q+1}
\end{pmatrix}
\begin{pmatrix}
r_1^q\\
1
\end{pmatrix}=0,
\]
whose determinant is $\N(\omega^{i}b+a-c)(\Tr(\omega^i b)-x_1^{q+1})-2a\Tr(\omega^i b)x_2^{q+1}$, or better $(\N(\omega^{i}b+a-c)-2a\Tr(\omega^i b))x_2^{q+1}$. If $\N(\omega^{i}b+a-c)\neq2a\Tr(\omega^i b)$, since $\Tr(\omega^i)\neq0$, the above Hermitian matrix is a non-singular matrix which defines a unitary form of $V(2,q^2)$ not admitting the point $\langle (1,0) \rangle$ as a totally isotropic point. Therefore, there are $q+1$ values for $r_1$ if the determinant is non-zero, otherwise there is a unique $r_1$ satisfying the above sesquilinear form. So $\N(\omega^{i}b+a-c)-2a\Tr(\omega^i b)=0$ needs to be studied. In this case, by writing the expression explicitly, we find that the elements $a\in\GF(q)^*$ 
making the determinant equal to zero are the solutions of the equation 
\begin{equation*}\label{eq_2bis}
X^2-(\Tr(\omega^i b)+2c)X+c^2-c\Tr(\omega^i b)+\N(\omega^i b)=0,
\end{equation*}
whose discriminant is
\begin{equation*}
\Delta=8 c\Tr(\omega^i b)-4\N(\omega^i b)+(\Tr(\omega^i b))^2.
\end{equation*}
Here, for $x_2\neq0$, by arguing exactly as before, the number of triples obtained is again $q^3-2q^2-q+1$, and so we may denote it by $p^4_{5 5}$.

Finally, we have $p^4_{5 4}=\eta_5-(p^4_{5 0}+p^4_{5 1}+p^4_{5 2}+p^4_{5 3}+p^4_{5 5})=q^4-3q^3+q^2+4q-1$.

Assume $k=5$. Then, $X=\langle (1,x_1,x_2,b) \rangle$ for some $b\in\GF(q)^*$, with $x_1^{q+1}+x_2^{q+1}=2b$. Therefore, $(R,X)\in R_5$ if and only if there is $c\in\GF(q)^*$ such that $r_1x_1^q+r_2x_2^q=a+ b-c$. As $b$ is fixed, we may set $d=b-c$ with $d\in \GF(q)\setminus \{b\}$. Assume $x_2=0$, i.e., $x_1^{q+1}=2b$. Thus $r_1=x_1^{-q}(a+ d)$, from which 
\[
r_2^{q+1}=-r_1^{q+1}+2a=-x_1^{-(q+1)}(a+ d)^2+2a.
\]
 This equation gives $r_2=0$  whenever the latter quantity is zero, otherwise it gives $q+1$ non-zero values for $r_2$. 
 Assume  first $d=0$. Then, $a^2-2 x_1^{q+1}a=0$ if and only if $a=2x_1^{q+1}=4 b$, thus $ r_1=4b/x_1^q$. Therefore, for $d=0$, we get only one  triple with $r_2=0$ and $(q-2)(q+1)$ triples with $r_2\neq0$.
 Now assume $d\neq0$.  The values of $a\in \GF(q)^*$ making $-r_1^{q+1}+2a=-x_1^{-(q+1)}(a+ d)^2+2a$ equal to zero  are the solutions of the equation $X^2+2(d-x_1^{q+1})X+d^2=0$, whose discriminant is $\Delta = x_1^{q+1}(x_1^{q+1}-2d)=4 bc\neq0$. Thus $\Delta$ is either a non-zero square or a non-square in $\GF(q)$. When $\Delta$ is a non-zero square, $x_1^{q+1}(x_1^{q+1}-2d)$ is a $\frac{q-1}{2}-$th root of unity, and for such a root, and hence for the corresponding $d$, there are two values of $a$ satisfying the previous quadratic equation in $X$. Note that for $d=0$, $x_1^{2(q+1)}$ is evidently a $\frac{q-1}{2}-$th root of unity (but $d=0$ has already been analysed before), while  $\Delta$ would be $0$ for $d=b$. Therefore, for any $\frac{q-1}{2}-$th root of unity but $x_1^{2(q+1)}$, we find two triples with $r_2=0$ and  $(q-3)(q+1)$ triples with $r_2\neq0$. For the remaining $\frac{q-1}{2}$ non-zero elements of $\GF(q)$, $\Delta$ is a non-square in $\GF(q)$. Thus, for any such an element we get $(q-1)(q+1)$ triples $(r_1,r_2,a)$.
 
By taking into account all the above results, for $x_2=0$, the following number of triples is obtained:
\begin{equation*}
1+(q-2)(q+1)+\frac{q-3}{2}(2+(q-3)(q+1))+\frac{q-1}{2}(q-1)(q+1)=q^3-2q^2+q+1.
\end{equation*}
In order to conclude the study of the case $k=5$, consider $x_2\neq0$. Therefore, by deriving $r_2$ from $r_1x_1^q+r_2x_2^q=a+ d$, where $d\in \GF(q)\setminus \{b\}$, and plugging it into $r_1^{q+1}+r_2^{q+1}=2a$, we get 
\[
r_1^{q+1}2 b-(a+d)\Tr(r_1 x_1^q)+(a+d)^2-2ax_2^{q+1}=0, 
\]
or
\begin{equation}\label{eq_3}
(r_1,1)\begin{pmatrix}
2 b & -(a+d)x_1^q\\
-(a+d)x_1 & (a+d)^2-2ax_2^{q+1}
\end{pmatrix}
\begin{pmatrix}
r_1^q\\
1
\end{pmatrix}=0,
\end{equation}
whose determinant is $((a+d)^2-4 a b)x_2^{q+1}$. If $(a+d)^2\neq4 a b$, the above Hermitian matrix is a non-singular matrix which defines a unitary form of $V(2,q^2)$ not admitting the point $\langle (1,0) \rangle$ as a totally isotropic point. Therefore, if $(a+d)^2-4 a b=0$, there is a unique $r_1$ satisfying \eqref{eq_3}, otherwise there are $q+1$ values for $r_1$ satisfying (\ref{eq_3}). It is evident that the $a\in\GF(q)^*$ making $(a+d)^2-4 a b$ equal to zero are the solutions of the equation  $X^2+2(d-2b)X+d^2=0$, whose discriminant is $\Delta=4bc\neq0$.  By arguing exactly as in the case $x_2=0$, we find that the number of triples obtained is again $q^3-2q^2+q+1$, and so we may denote it by $p^5_{5 5}$.

Finally, we get $p^5_{5 4}=\eta_5-(p^5_{5 0}+p^5_{5 1}+p^5_{5 2}+p^5_{5 3}+p^5_{5 5})=q^4-3q^3+q^2+2q-1$.
\end{proof}

\begin{lemma}\label{lem_6}
The intersection numbers $p_{4j}^k$ are well defined. They are collected in the following intersection matrix $L_4$ whose $(k,j)-$entry is $p_{4j}^k$
{
\[
L_4=\scriptsize\begin{pmatrix}
0 & 0 & 0 & 0        & (q^3-q)(q-1)^2          & 0 \\
0 & q(q-1)^2 & 0   & q(q-1)^3       & q^2(q-1)^2(q-2)   &q(q-1)^3 \\
0 & 0 & 0 & 0 &  (q^3-q)(q-1)(q-2)        & (q^3-q)(q-1)  \\
0 & q(q-1)^2 & 0 & q(q-1)^3 & q(q-1)(q^3-3q^2+2q+1)    & q^2(q-1)(q-2)\\
1 & q^3-q^2-2q & q-2 &  q^4-2q^3-q^2+3q+1 & q^5-4q^4+4q^3+3q^2-7q+1   & q^4-3q^3+q^2+4q-1 \\
0 & (q+1)(q-1)^2 & q-1 & q(q^2-1)(q-2)  & q^5-4q^4+4q^3+3q^2-5q+1 & q^4-3q^3+q^2+2q-1 
 \end{pmatrix}
\]}
\end{lemma}
\begin{proof}
To check that  $p_{4j}^k$ is well defined, for any pair $(X,Q)\in R_k$ we would have to count the number of points $R$ that are $4-$related  with $Q$ and  $j-$related with $X$. However, thanks to the previous Lemmas \ref{lem_2}, \ref{lem_3}, \ref{lem_4}, \ref{lem_5}, we already have all the entries of $L_4$ except for $p_{4 4}^k$, with $k=1,\ldots,5$. These can be derived from the well-known relations $\eta_i=\sum_{j}p_{i j}^k$, i.e., $p^k_{4 4}=\eta_4-(p^k_{4 0}+p^k_{4 1}+p^k_{4 2}+p^k_{4 3}+p^k_{4 5})$, for $k=1,\ldots,5$.
\end{proof}

Set $I=\{0,2\}$. For $i,j\in I$ and $k\in \{0,\ldots,5\}\setminus I$, we have $p^{k}_{i,j}=0$. By Theorem 9.3 (iii) in \cite{bi}, the scheme $\mathfrak X_P$ is imprimitive. 

Summarising, we have the following result.
\begin{theorem}\label{prop_1}
 For any  point $P$ of $H(3,q^2)$, $\mathfrak X_P=(\cX,\cR)$ is a symmetric and  imprimitive association scheme, whose first and second eigenmatrices are  
{\small
\[
\cP=\begin{pmatrix}
1 & (q^2-1)(q+1) & q-1 & (q^2-1)^2  & (q^3-q)(q-1)^2 & (q^3-q)(q-1)\\
1 & q^2-q-1      & q-1 & q^3-2q^2+1 & -q(q-1)^2        & -q(q-1) \\
1 & q^2-1        & -1  & -q^2+1     & 0                & 0 \\
1 & -q-1         & q-1 & -q^2+1     & q(q-1)           & q\\
1 & -q-1         & -1  & q+1        & -q(q+1)          & q(q+1)  \\
1 & -q-1         & -1  &  q+1       & q(q-1)           & -q(q-1) \\
 \end{pmatrix}
\]}
\vspace{.1in}
{\footnotesize
\[
\cQ=\begin{pmatrix}
1 & (q-1)(q+1) ^2 & q^3(q-1) & q(q-1)^2(q+1)            &  \frac12 q^2(q-1)^3    & \frac12q^2 (q+1)(q-1)^2                       \\

1 & q^2-q-1 & \frac{q^3(q-1)}{q+1}   & -q(q-1) & -\frac{q^2(q-1)^2}{2(q+1)}   & -\frac12q^2(q-1)            \\

1 &  (q-1)(q+1)^2 & -q^3 & q(q-1)^2(q+1) & -\frac12q^2(q-1)^2         & -\frac12q^2(q+1)(q-1)             \\

1 & q^2-q-1 & -\frac{q^3}{q+1}   & -q(q-1) & \frac{q^2(q-1)}{2(q+1)}      & \frac{q^2}{2} \\

1 &-q-1 & 0 &  q         &  -\frac{q^2}{2}              & \frac{q^2}{2}  \\

1 & -q-1 & 0   & q            &\frac12q^2(q-1)           & -\frac12q^2(q-1)         
 \end{pmatrix}
\]
}
\end{theorem}
\begin{remark}
We point out that $\mathfrak X_P=(\cX,\cR)$ is a fusion of a Schurian scheme. Let $G$ be the collineation group of $H(3,q^2)$ and let $G_P$ be the stabiliser of the point $P$. Then, it is not difficult to see that the action of $G_P$ on the set $\cX$ of points not collinear with $P$ is generously transitive. Hence, the permutation group $G_P^{\cX}$ induced by $G_P$ acting on $\cX$ yields a Schurian scheme. In fact, this association scheme has $4+(q-1)/2$ classes, since the relation $R_4$ splits up under the action of $G_P^{\cX}$. Hence, an alternative proof of Theorem \ref{prop_1} could be given that would rely on understanding the irreducible constituents of the permutation character $G_P^{\cX}$ and computation of the intersections of double cosets $HgH\cap k^{G_P}$, where $H$ is a point stabiliser in $G_P^{\cX}$ and $k^{G_P^{\cX}}$ is a conjugacy class of elements of $G_P^{\cX}$.
\end{remark}

\subsection*{A quotient scheme}

We now explore the quotient scheme, say $\mathfrak{B}$, arising from the imprimitivity of $\mathfrak X_P$.  
 Since $R_0\cup R_2$ is an equivalence relation on $\cX$, the vertices of $\mathfrak{B}$ are  the hyperbolic lines through $P$; this set will be denoted by $\Sigma$. Let $\sim$ be the equivalence relation on $\{0,\ldots,5\}$ defined by 
 \[
 i\sim j \mathrm {\ \ \ \  \ if\ and\ only\ if\  \ \ \ \ } p^j_{i\alpha}\neq 0,  \mathrm {\ \ \ \ \ \ \ \ \ for\ some\ } \alpha\in I.
 \] 
 
 The equivalence classes are $I=\{0,2\}$, $1'=\{1,3\}$, $2'=\{4,5\}$. This yields that the non-trivial relations of $\mathfrak{B}$ are
 \[
 \begin{array}{lcl}
 R_{1'} & = & \{([x],[y])\in \Sigma\times \Sigma: (x,y)\in R_1\cup R_3\},\\[.05in]
 R_{2'} & = & \{([x],[y])\in \Sigma\times \Sigma: (x,y)\in R_4\cup R_5\},
 \end{array}
 \] i.e., $(\Sigma, R_{1'})$ is a strongly regular graph (and the same holds for its complementary graph $(\Sigma, R_{2'})$).
 
Since $\sim$ is an equivalence relation, the set $R_1(x)\cup R_3(x)$ is partitioned into hyperbolic lines, for any $x\in \cX$. Then, the valency of the graph is $k=(\eta_1+\eta_3)/q=(q^2-1)(q+1)$. Similarly,  the set 
\[
\underset{a,b=1,3}\bigcup\cP^{(x,y)}_{a,b},
\]
where  $\cP^{(x,y)}_{a,b}=\{z\in\cX:(x,z)\in R_a, (y,z)\in R_b\}$, for $(x,y)\in R_i$, $i\notin I$,  is partitioned into hyperbolic lines. This implies that the other parameters of $(\Sigma, R_{1'})$ are
\[
\lambda=\frac{p^1_{11}+2p^1_{13}+p^1_{33}}q=\frac{p^3_{11}+2p^3_{13}+p^3_{33}}q=2q^2-q-2
\]
 and 
 \[
 \mu=\frac{p^4_{11}+2p^4_{13}+p^4_{33}}q=\frac{p^5_{11}+2p^5_{13}+p^5_{33}}q=q(q+1).
 \]
Let $\Bil_{2}(q)$ be the graph defined on the set of bilinear forms from $V(2,q)\times V(2,q)$ to $\GF(q)$,  with two forms $f$ and $g$  being adjacent if and only if $\rank(f-g)=1$. By \cite[Proposition 2.6]{cs}, $\Bil_{2}(q)$ is isomorphic to the matrix algebra $\cD_{2}(q^2)$ consisting of all $2\times 2$ Dickson matrices over $\GF(q^2)$, where  a $2\times 2$ Dickson  matrix over $\GF(q^2)$ has the form
 \[
D_{(a,b)}= \left(\begin{array}{cccc}
a     & b \\
b^q & a^q
\end{array}\right),
 \]
with $a,b\in\GF(q^2)$.

\begin{proposition}\label{prop_6}
The graph $(\Sigma, R_{1'})$ is isomorphic to $\Bil_{2}(q)$.
\end{proposition}
\begin{proof}
Any point $\langle (1,r_1,r_2,r_3) \rangle$ with $r_3+r_3^q=r_1^{q+1}+r_2^{q+1}$, together with $P=\langle (0,0,0,1) \rangle$,  spans a hyperbolic line of $H(3,q^2)$, denoted by $l_{(r_1,r_2)}$. Fix $\delta\in\GF(q^2)$ with $\N(\delta)=-1$, and define the  bijection
\[
\begin{array}{rccc}
\varphi:& \Sigma & \rightarrow &  \Bil_{2}(\GF(q))\\
 & l_{(r_1,r_2)} & \mapsto  & D_{(r_1, \delta r_2)}
\end{array}.
\]
Let $m=l_{(m_1,m_2)},n=l_{(r_1,r_2)}\in\Sigma$. For any $Q\in m$ and $R\in n$, 
\mbox{$z(P,Q,R)=h(\q,\r)\GF(q)^*$}. Straightforward calculations show that $(m,n)\in R_{1'}$ if and only if \mbox{$\Tr(h(\q,\r))=0$}, i.e., $\det (\varphi(m)-\varphi(n))=0$. 
\end{proof}

\section{Pseudo-ovals as subsets in the  scheme $\mathfrak X_P$ on $Q^-(5,q)$}\label{sec_4}

Via the Klein correspondence $\kappa$, the lines of $\PG(3,q^2)$ are mapped to the points of a hyperbolic quadric  $Q^+(5,q^2)$  of $\PG(5,q^2)$.  In particular, the lines of a unitary polar geometry of rank 2 of $\PG(3,q^2)$  are mapped to the points of an elliptic quadric   $Q^-(5,q)$ of a $\PG(5,q)$   embedded in $\PG(5,q^2)$. The reader is referred to  \cite{hir} for more details on the  Klein correspondence.

Let 
\[
\begin{array}{rccc}
\tau:& V(6,q^2) & \longrightarrow  & V(6,q^2) \\[.1in]
& (X_1,X_2,X_3,X_4,X_5,X_6)& \mapsto  & (X_1,-\mu^{q}X_2,X_3,-\mu^{q}X_4,X_5,\mu^{q}X_6),
\end{array}
\] 
for a fixed $\mu\in\GF(q^2)$ with $\N(\mu)=-1$, and set $\rho=\tau \circ \kappa$. It turns out that  the  lines of $H(3,q^2)$ defined by \eqref{eq_1} are mapped by $\rho$ to the points of the $Q^-(5,q)$ defined by  $\widehat Q$ in \eqref{eq_18}.
For any point $P\in H(3,q^2)$, by abuse of notation, we write $\rho(P)$ to denote the  totally singular line 
$
\{\rho(r):r\mathrm{\ is\ a\ totally\ isotropic\ line\ on\ }P\}
$.
 
 \begin{proposition}\label{prop_3}
Let $P$, $Q$ and $R$ be three distinct points of the $H(3,q^2)$ defined by \eqref{eq_1}. Then, 
\mbox{$z(P,Q,R)=e$} if and only if $l=\rho(P)$, $m=\rho(Q)$, $n=\rho(R)$ 
are in perspective.
\end{proposition}
\begin{proof}
Let $P$, $Q$ and $R$ such that $z(P,Q,R)=e$. 
Since $\langle P,Q,R \rangle$ is a non-degenerate plane, coordinates can be chosen such that  $P=\langle (1,0,0,0) \rangle$, $Q=\langle (0,0,0,1) \rangle$ and $R=\langle (1,t_1,t_2,t_3) \rangle$, with $t_3=\frac12(t_1^{q+1}+t_2^{q+1})\neq0$. 

The totally isotropic lines on $P$ have the form $\langle (1,0,0,0),(0,x,\mu x^q,0) \rangle$, for some non-zero $x\in\GF(q^2)$. Thus, $\rho(P)=L(I,0,0)$. Similarly, $\rho(Q)=L(0,0,I)$.

The totally isotropic lines on $R$ have the form $\langle (1,t_1,t_2,t_3),(0,x,\mu x^q,t_1^qx+t_2^q\mu x^q) \rangle$, for some non-zero $x\in\GF(q^2)$. Thus, $\rho(R)=L(I,F_1,F_2)$, with $F_1(x)=t_1^qx+t_2\mu x^q$ and $F_2(x)=\frac12(t_1^{q+1}-t_2^{q+1})x+t_1t_2^q\mu x^q$.
By using Proposition \ref{prop_2}, it is immediate that $l$, $m$ and $n$ 
are in perspective.

 Let $l$, $m$, $n$  be three lines of the $Q^-(5,q)$ arising from $\widehat Q$ which  are in perspective. Set  $P=\rho^{-1}(l)$, $Q=\rho^{-1}(m)$, $R=\rho^{-1}(n)$. We proceed to show that $z(P,Q,R)=e$.
 
 Coordinates in $\widehat V$ can be chosen such that
 \[
l=L(I,0,0), \ \ \ m=L(0,0,I), \ \ \  n=L(I,F_1,F_2),
\]
with $F_1,F_2$ non-singular. Let $\overline n$ be the extension of $n$ in $V(6,q^2)$, that is
\begin{equation*}\label{eq_9}
\overline n=\{(x,y, f_1x+g_1y,g_1^qx+f_1^qy,f_2x+g_2y,g_2^qx+f_2^qy):x,y\in\GF(q^2)\}.
\end{equation*}
Note that $P=\rho^{-1}(l)=\langle (1,0,0,0) \rangle$, $Q=\rho^{-1}(m)=\langle (0,0,0,1) \rangle$. Furthermore, we find that $\rho^{-1}(\overline n)$ actually consists of $q^2+1$ coplanar lines of $\PG(3,q^2)$ through $R$, $q+1$ of them are totally isotropic. Since $m$ and $n$ are not concurrent, then we may write $R=\langle (1,t_1,t_2,t_3 \rangle$, with $t_3^q+t_3=t_1^{q+1}+t_2^{q+1}$. Straightforward calculations show that these lines have the form $\langle (1,t_1,t_2,t_3),(0,x_1,x_2,t_1^qx_1+t_2^q x_2) \rangle$, for some $x_1,x_2\in\GF(q^2)$. For $x_1=1$ and $x_2=0$, the $\rho-$image of the corresponding line is $\langle (1,0,t_1^q,\mu^qt_2, t_1^{q+1}-t_3,\mu^q t_1^qt_2) \rangle\in\overline n$. On the other hand, the unique point of $\overline n$ of this form is $\langle (1,0,f_1,g_1^q,f_2,g_2^q) \rangle$, whence 
\[
\left\{
\begin{array}{lcl}
t_1^q & = & f_1 \\
\mu ^q t_2 & = &  g_1^q\\
t_1^{q+1}-t_3 & = & f_2\\
\mu^q t_1^qt_2 & = & g_2^q\\
t_3^q+t_3 & = & t_1^{q+1}+t_2^{q+1}
\end{array}
\right..
\]
Eqs. \eqref{eq_5} with $F_0=I$ make the system compatible, and so we get the unique solution $t_1 = f_1^q$, 
 $t_2=-\mu  g_1^q$ and $t_3=f_1^{q+1}-f_2$.
 By Proposition \ref{prop_2}, $f_2\in\GF(q)$ providing $t_3\in\GF(q)$. Then, $z(P,Q,R)=t_3\GF(q)^*=e$, which is the desired conclusion. 
 \end{proof}
Proposition \ref{prop_3}  allows us to view the association scheme $\mathfrak X_P$ in
the dual setting. Fix a line $l$ of $Q^-(5,q)$ and consider the set $\cX'$ of all lines of $Q^-(5,q)$ that are disjoint from $l^\perp$, that is, from $l$.
There are $q^5$ such lines and we equip this set with the following five non-trivial relations:
\begin{enumerate}
\item[] \mbox{$R'_1=\{(m,n): m\ \mathrm{and\ } n   \mathrm{\ are \ concurrent}\}$},
\item[] \mbox{$R'_2=\{(m,n):\dim\,\langle l, m, n \rangle=4\}$},
\item[] \mbox{$R'_3=\{(m,n):m\ \mathrm{and\ } n \mathrm{\ are\ disjoint\ and\ }\dim\,\langle l, m, n \rangle=5\}$},
\item[] \mbox{$R'_4=\{(m,n):l,m,n$  span the whole space and they are  not  in  perspective$\}$},
\item[] \mbox{$R'_5=\{(m,n):l,m,n$  span the whole space and they are  in   perspective$\}$}.
\end{enumerate}

By using the well-known correspondences given by the Klein map from $H(3,q^2)$ to $Q^-(5,q)$, we get that $\{(\rho(P),\rho(Q)): (P,Q)\in R_i\}=R'_i$, for $i=1,2,3$. Proposition \ref{prop_3} provides the  equivalence between $R_5$ and $R_5'$, hence the one between $R_4$ and $R_4'$.  The transitivity of the unitary group on the points of $H(3,q^2)$, as well as the transitivity of the orthogonal group on the lines of $Q^-(5,q)$,  leads to the following result.
\begin{theorem}
Set $\cR'=\{R_0',R_1',\ldots,R_5'\}$, where $R_0'$ is the diagonal relation. $\mathfrak X_l=(\cX',\cR')$ is a symmetric, imprimitive association scheme, isomorphic to \mbox{$\mathfrak X_P=(\cX,\cR)$}, for any line $l$ of $Q^-(5,q)$ and any point $P$ of $H(3,q^2)$.
\end{theorem}

Let $\{A_i\}_{0\le i\le d}$ be the adjacency matrices for a $d-$class association scheme $\mathfrak X=(\cX,\{R_i\}_{0\le i\le d})$,  and let  $\{E_i\}_{0\le i\le d}$ be the set of minimal idempotents for $\mathfrak X$.
For any subset $Y$  of $\cX$, $\chi_Y$ will denote the characteristic vector of $Y$. 

The {\em inner distribution} of a non-empty subset $Y$ of $\cX$ is the array $\ba=(a_0,\ldots, a_d)$ of the non-negative rational numbers $a_i$ given by
\[
a_i=|Y|^{-1}|R_i\cap Y^2|=|Y|^{-1}\chi_Y A_i \chi_Y^\top.
\]

Let $M$ be a subset of $\{0,\ldots,d\}$ with $0\in M$. A non-empty subset $Y$ of $\cX$ is an $M-${\em clique} of $\mathfrak X$ if it satisfies
\[
\mbox{$R_i\cap Y^2=\varnothing$, \ \ \ \ \ \ \ for all $i\in \{0,\ldots,d\}\setminus M$,} 
\] 
or equivalently, the $i-$th entry of the inner distribution $\ba$ of $Y$ is zero, for all $i\in \{0,\ldots,d\}\setminus M$. 
Let $T$ be a subset of $\{1,\ldots,d\}$. A non-empty subset $Y$ of $\cX$ is a $T-${\em design} of $\mathfrak X$ if its inner distribution $\ba$ satisfies 
\[
\sum_i{a_i\cQ(i,j)}=0, \ \ \ \ \ \mathrm{\ for\ all\ } j\in T,
\] 
where $\cQ$ is the second eigenmatrix of the scheme. Equivalently, $Y$ is a $T-$design if and only if $\chi_YE_j= 0$, for all $j\in T$.

The \emph{dual degree set} of a vector $v \in \mathbb{R}^{|\cX|}$ is the set of indices $j\in\{1,\ldots, d\}$ such that $vE_j\ne 0$.
Two vectors of $\mathbb{R}^{|\cX|}$ are \emph{design-orthogonal} if their dual degree sets are disjoint.

Recall that a {\em pseudo-oval} of $\PG(5,q)$ is a set $\cS$ of $q^2+1$ lines, such that any three distinct elements of $\cS$ span the whole space.  We consider pseudo-ovals consisting only of lines of  ${Q}^-(5,q)$.

By transferring on $H(3,q^2)$ the characterization of pseudo-conics of $Q^-(5,q)$ by Thas \cite[Theorem 6.4]{thas}, the characterization of Cossidente, King and Marino \cite{ckm} is obtained as a corollary of Proposition \ref{prop_3}.  
\begin{corollary}[{\cite[Theorem 3.1]{ckm}}]
A special set $\tilde \cS$ of $H(3,q^2)$  is of CP-type if and only if $z(P,Q,R)=e$, for all  triples of distinct points $P,Q,R$ of $\tilde\cS$.
\end{corollary}
\begin{proof}
By Theorem 2.1 in \cite{cmp}, a special set $\tilde\cS$ of CP-type corresponds to a pseudo-conic $\cS$ of $Q^-(5,q)$ under the Klein map $\rho$. By \cite[Theorem 6.4]{thas}, this means that any three distinct elements $l,m,n$ of $\cS$  are in perspective, that is, by Proposition \ref{prop_3}, $z(P,Q,R)=e$, where $P=\rho^{-1}(l)$, $Q=\rho^{-1}(m)$ and $R=\rho^{-1}(n)$. We note that $z(P,Q,R)=e$ is equivalent to the fact that the Segre invariant of $(P,Q,R)$ defined in \cite{ckm} is equal to 1.
\end{proof}
\begin{proposition}\label{lem_18}
Let $\cS$ be a set of $q^2+1$ lines of $Q^-(5,q)$. Then, $\cS$ is a pseudo-oval if and only if every non-degenerate hyperplane contains  either $0$ or $2$ elements of $\cS$.
\end{proposition}
\begin{proof}
Let $\cS$ be a pseudo-oval. There are  $q^2(q^3+1)$ non-degenerate hyperplanes in a $\PG(5,q)$,  among which  $q^2(q+1)$ contain a given totally singular line. 
A simple double count shows that the number of non-degenerate hyperplanes containing a pair of disjoint totally singular lines is $q+1$.  
Now count the triples $(l,m,\Pi)$ where $l$ and $m$ are distinct totally singular lines of $\cS$,  $\Pi$ is a non-degenerate hyperplane, under the conditions that $l$ and $m$ are disjoint and $\Pi$ contains $\langle l, m \rangle$. For any non-degenerate hyperplane $\Pi_i$, let $\mu_i$ be the number of elements of $\cS$ contained in $\Pi_i$. 

Then, we have
\begin{align*}
\sum_i \mu_i(\mu_i-1)&= |\cS| (|\cS|-1) (q+1)\\
&=(q^2+1)q^2(q+1).
\end{align*}
On the other hand, the number of  pairs $(l,\Pi_i)$, with $l\in\cS$ contained in $\Pi_i$, is
\begin{align*}
\sum_i \mu_i&= |\cS| q^2(q+1)\\
&= (q^2+1)q^2(q+1).
\end{align*}
Since the two sums are equal, it follows that  
\[
\sum_i \mu_i(2-\mu_i)=\sum_i\mu_i - \sum_i \mu_i(\mu_i-1)=0.
\]
Every three elements of $\cS$ span the whole space, so $\mu_i\le 2$ for each $i$.
Therefore, each term of the left-most sum is  positive,  hence $\mu_i(2-\mu_i)=0$
for each $i$, i.e., $\mu_i\in\{0,2\}$.

Conversely, let $l,m,n$ be three lines of $\cS$. Assume that $l$ and $m$ intersect. Simple geometric arguments show that $\langle l,m,n\rangle$ is a 4-dimensional subspace which is contained in some non-degenerate hyperplane. This contradicts the property of $\cS$.  Assume that $l$ and $m$ are disjoint and $n$ intersects  $\langle l,m\rangle$ in a point. Then, $\langle l,m,n\rangle$ is a non-degenerate hyperplane, and we have again a contradiction. Therefore, $\cS$ is a pseudo-oval.
\end{proof}

\begin{remark}
The ``\ if\ '' part of Proposition \ref{lem_18} was already proved in  \cite{pt}, see result 8.7.2. To check this,   note that a hyperplane containing $l^\perp$, for some totally singular line $l$, is degenerate; and conversely. 
\end{remark}
\begin{theorem}\label{lem_13}
Let $\cS$ be a pseudo-oval of $Q^-(5,q)$. Then
\begin{itemize}
\item[(a)] $\cS\setminus\{l\}$ is a $\{0,4,5\}-$clique of $\mathfrak X_l$, and
a $\{1\}-$design of $\mathfrak X_l$, for each $l\in \cS$.
\item[(b)] The following are equivalent:
\begin{itemize}
\item[(i)] $\cS\setminus\{l\}$ is a $\{0,5\}-$clique, for each $l\in \cS$;
\item[(ii)] $\cS\setminus\{l\}$ is a $\{1,5\}-$design, for each $l\in \cS$;
\item[(iii)] $\cS$ is a pseudo-conic;
\end{itemize}
\end{itemize}
\end{theorem}

\begin{proof}
Let $l$ be any  line of $\cS$ and set $\cS'=\cS\setminus\{l\}$. By the definition of pseudo-oval and the scheme  $\mathfrak X_l$, we find that $\cS'$ is a $\{0,4,5\}-$clique of $\mathfrak X_l$.

Let $\mathbf{a}$ be the inner distribution of $\cS'$ (note that $|\cS'|=q^2$):
\[
\mathbf{a} =\frac{1}{q^2}\left( \chi_{\cS'} A_i \chi_{\cS'}^\top \right)_{i=0}^5=(1,0,0,0,x,q^2-x-1),
\]
where $x$ is undetermined.   
The MacWilliams Transform $\ba \cQ$ of $\mathbf{a}$ is
\begin{equation}\label{eq_19}
\ba \cQ=\left(q^2,0,q^3(q-1),q^2(q^2-1),\frac12q^3(2q^2-4q+2-x),\frac{q^3x}{2}\right).
\end{equation}
Therefore, $\cS'$ is a $\{1\}-$design, and (i) and (ii) are equivalent in (b). 
To see the equivalence with  (iii) in (b), note  that $\cS'$ is a $\{0,5\}-$clique with respect to every $l$ in $\cS$  if and only if $l,m,n$ are in perspective, for every triple $l,m,n$ of distinct lines of $\cS$, i.e.,   $\cS$ is a pseudo-conic by \cite[Theorem 6.4]{thas}.
\end{proof}
Let $\cU_{p_1,p_2}=\cO_{1}\cup \cO_{2}$ be any set constructed as in Section \ref{sec2U}, and $\Pi$ the hyperplane containing it. Let $\chi_{\cO_i}$  be the characteristic vectors of $\cO_i$, , $i=1,2$. We will see (Proposition \ref{prop_4}) that $v=\chi_{\cO_1}-\chi_{\cO_2}$ and the characteristic vector of a pseudo-conic are design-orthogonal.

We introduce the following subsets of $\cX$ referred to $\cU_{p_1,p_2}$:
\begin{itemize} 
 \item[-] $V$  is the set of lines of $\cX$ contained in $\Pi$ and not intersecting $p_1\cup p_2$; 
 \item[-] $J_i$ is  the set of lines of $\cX$ not contained in $\Pi$ and intersecting $p_i$, $i=1,2$;  
 \item[-] $W$ is the set of lines of $\cX$ not contained in $\Pi$ and  intersecting $(B^\perp\cap\Pi) \setminus (p_1\cup p_2)$;  
 \item[-] $Z$ is the set of lines of $\cX$ not contained in $\Pi$ and  not intersecting $B^\perp\cap\Pi$.
 \end{itemize}
\begin{lemma}\label{lem_11}\leavevmode
\begin{itemize}
\item[] $\chi_{\cO_1}A_1=\j+ (q-2)\chi_{\cO_1}+(q-1)(\chi_{\cO_2\cup V}+\chi_{J_1})-\chi_{J_2\cup W}$;
\item[] $\chi_{\cO_1}A_2= \j-\chi_{\cO_1}-\chi_{\cO_2\cup V}-\chi_{J_2\cup W}-\chi_{Z}$;
\item[] $\chi_{\cO_1}A_3=  \j+(q^2-q-1)\chi_{\cO_1}-\chi_{\cO_2\cup V}+(q^2-q-2)\chi_{J_1}+(q-1)\chi_{J_2\cup W}+(q-2)\chi_{Z}$;
\item[] $\chi_{\cO_1}A_4=(q^2-1)(\j-\chi_{\cO_1}-\chi_{\cO_2}-\chi_{J_1})-(q-1)(\chi_{J_2}+\chi_V+2\chi_Z)-(2q-1)\chi_{W}$;
\item[] $\chi_{\cO_1}A_5=\j-\chi_{\cO_1}+(q^2-q-1)\chi_{\cO_2}-\chi_{J_1}-\chi_{J_2}+(q-1)\chi_W+(q-2)\chi_{Z}$,
\end{itemize}
where $\j$ is the all-ones vector. Similarly, for $\chi_{\cO_2}$.
\end{lemma}
\begin{proof}
We calculate $\chi_{\cO_1}A_1$. This is equivalent to counting how many lines of $\cO_1$ are concurrent with a fixed line $n\in\cX$. 

Assume $n\in \cO_1$. Then,  there are $q-1$ lines of $\cO_1$ concurrent with $n$. 
Assume $n\in \cO_2\cup V$. For each point of $p_1\setminus\{B\}$ there is exactly one line of $\cO_1$ intersecting $n$. So, we find $q$ lines of $\cO_1$ concurrent with $n$.
Assume $n\in J_1$. Set $R=n\cap p_1$. Then, the unique lines of $\cO_1$ concurrent with $n$ are those through $R$, which are $q$. Assume $n\in J_2\cup W$. Set $R=n\cap B^\perp$. The unique line joining $R$ and $p_1$ is $\langle B,R \rangle$, that is not in $\cX$. In this case $n$ contributes 0.
Assume $n\in Z$. Set $R=n\cap \Pi$. There is a unique line joining $R$ and $p_1$, and  it is  in $\cO_1$. In this case $n$ contributes 1. Finally, 
\[
\begin{array}{rcl}
\chi_{\cO_1}A_1 & = & (q-1)\chi_{\cO_1}+q\chi_{\cO_2\cup V}+q\chi_{J_1}+(\j-\chi_{\cO_1}-\chi_{\cO_2\cup V}-\chi_{J_1}-\chi_{J_2\cup W})\\[.05in]
         & = & \j+ (q-2)\chi_{\cO_1}+(q-1)(\chi_{\cO_2\cup V}+\chi_{J_1})-\chi_{J_2\cup W}.
\end{array}
\]

We now compute $\chi_{\cO_1}A_2$. This is equivalent to counting how many lines of $\cO_1$ are contained in the  4-dimensional subspace $\langle l,n \rangle$, for a fixed $n\in\cX$. 

Assume $n\in \cO_1$. Since the plane $\langle l,n \rangle\cap \Pi$, containing $n$ and $p_1$, is degenerate, there are no lines of $\cO_1$ different from $n$ satisfying the property.
Assume $n\in \cO_2\cup V$. By arguing as above, it is easy to see that there are no lines of $\cO_1$ different from $n$ satisfying the property in this case too.
Assume $n\in J_1$. The plane $\langle l,n \rangle\cap \Pi$ is degenerate containing $p_1$. Thus, it contains exactly a further totally singular line which is necessarily in $\cO_1$.  
Assume $n\in J_2\cup W$. By arguing as above, the plane $\langle l,n \rangle\cap \Pi$ is degenerate as it contains a totally singular line $p$ on $B$. Then,  the plane contains exactly a further totally singular line intersecting $p$, which is not in $\cO_1$. 
Assume $n\in Z$. Since the plane $\langle l,n \rangle\cap \Pi$  is non-degenerate, there are no lines of $\cO_1$ satisfying the property. 
Summarising, 
\[
\begin{array}{rclcl}
\chi_{\cO_1}A_2 & = & \chi_{J_1} & = & \j-\chi_{\cO_1}-\chi_{\cO_2\cup V}-\chi_{J_2\cup W}-\chi_{Z}.
\end{array}
\]

We now compute $\chi_{\cO_1}A_3$. This is equivalent to counting how many lines of $\cO_1$ share a point with the  4-dimensional subspace $\langle l,n \rangle$ which is not on $n$, for a fixed $n\in\cX$. 

Assume $n\in \cO_1$. Since the plane $\langle l,n \rangle\cap \Pi$, containing $n$ and $p_1$, is degenerate, there are $q(q-1)$ lines of $\cO_1$ sharing one point with $p_1$, different from $n\cap p_1$. 
Assume $n\in \cO_2\cup V$.  By arguing as above, it is easy to see that there are no lines of $\cO_1$  satisfying the property.
Assume $n\in J_1$.  The plane $\langle l,n \rangle\cap \Pi$,  containing $p_1$, is degenerate.  Then, it contains exactly a further totally singular line which is necessarily in $\cO_1$.  Hence, there are $(q-2)q+q-1=q^2-q-1$ lines of $\cO_1$ that intersect $p_1$.
Assume $n\in J_2\cup W$. The plane $\langle l,n \rangle\cap \Pi$ is degenerate as it contains a totally singular line $p$ on $B$ and a further totally singular line, say $s$, intersecting $p$. 
For any $R\in p_1\setminus\{B\}$, there is a unique line of $\cO_1$ on $R$ concurrent with $s$. Hence, there are $q$ lines of $\cO_1$ that intersect $\langle l,n \rangle$ in a point not in $n$.
Assume $n\in Z$. The plane $\langle l,n \rangle\cap \Pi$  is non-degenerate, and, together with $p_1$, it spans a 4-dimensional subspace  intersecting $Q^-(5,q)$ in either a hyperbolic quadric or  a quadratic cone projecting the conic in the plane  from a point of $p_1\setminus\{B\}$. Anyway, the number of lines of $\cO_1$ that meet $\langle l,n \rangle$ in exactly one point not in $n$ is $q-1$.
Therefore, 
\[
\begin{array}{rclcl}
\chi_{\cO_1}A_3 & = & q(q-1)\chi_{\cO_1}+(q^2-q-1)\chi_{J_1}+q\chi_{J_2\cup W}+(q-1)\chi_{Z}\\[.05in]
        & = & \j+(q^2-q-1)\chi_{\cO_1}-\chi_{\cO_2\cup V}+(q^2-q-2)\chi_{J_1}+(q-1)\chi_{J_2\cup W}+(q-2)\chi_{Z}.
\end{array}
\]

We now compute $\chi_{\cO_1}A_5$. This is equivalent to counting how many lines $m$ of $\cO_1$ span the whole space  together with $l$ and a fixed line $n\in\cX$, such that $l,m,n$ are in perspective. 

For any $n\in \cO_1$, there is no line of $\cO_1$ satisfying the property.
Assume $n\in \cO_2$.  By the arguments used to calculate $\chi_{\cO_1}A_1$, there are $q$ lines of $\cO_1$ which are concurrent with $n$. All the other $q^2-q$ lines of $\cO_1$ satisfy the property by Lemma \ref{lem_12}.  If $n\in  V$, there is no line of $\cO_1$ satisfying the property by Lemma \ref{lem_12}.
For any $n\in J_1$, there is no line of $\cO_1$ satisfying the property. 
Assume $n\in J_2$. For any given line $m\in\cO_1$, by Lemmas \ref{lem_8} (iii) and \ref{lem_12}, the lines in $\cX$ that satisfy the property are those contained in the  unique hyperplane $\Pi$ such that $p_1$ and $p_2$ correspond under the involution $\tilde\sigma$. Therefore, for any line $n\in J_2$, there are no lines in $\cO_1$ satisfying the property.
Assume $n\in W$. Let $p\neq p_2$ the unique totally singular line on $B$ concurrent with $n$. 
By Lemma \ref{lem_12},  $m\in\cO_1$ satisfies the property if and only if $p$ corresponds to $p_1$ under the involution arising from some non-degenerate hyperplane containing $p_1$ and $p$ but not $l$. Let $\Lambda$ be such a hyperplane. The 4-dimensional subspace $\Lambda\cap\Pi$, containing $p$ and $p_1$,  meets $Q^-(5,q)$ in either a quadratic cone or a hyperbolic quadric. If the former case occurred,  $\Lambda$ and $\Pi$ would define the same line $\sigma$ by Remark \ref{rem_2}, and then $p=p_2$. Hence, the intersection is necessarily a hyperbolic quadric. This implies that the number of lines of $\cO_1$ satisfying the property is $q$.
Assume $n\in Z$. Let $p$ the unique totally singular line on $B$ concurrent with $n$. 
By Lemma \ref{lem_12},  $m\in\cO_1$ satisfies the property if and only if $p$ corresponds to $p_1$ under the involution arising from some non-degenerate hyperplane containing $p_1$ and $p$ but not $l$. Let $\Lambda$ be such a hyperplane. The 4-dimensional subspace $\Lambda\cap\Pi$, containing $p_1$,  meets $Q^-(5,q)$ in the line $p_1$, a quadratic cone or a hyperbolic quadric. If the former case occurred, $\Lambda\cap\Pi=p_1^\perp$ from which  $\Lambda$ and $\Pi$ would be degenerate (as $\Lambda^\perp,\Pi^\perp\in p_1$). Hence, the intersection is necessarily a quadratic cone or a hyperbolic quadric. In each case, there is exactly one line of $\cO_1$ concurrent with $n$, and so there are $q-1$ lines satisfying the property.

Finally, 
\[
\begin{array}{rclcl}
\chi_{\cO_1}A_5 & = & (q^2-q)\chi_{\cO_2}+q\chi_{W}+(q-1)\chi_{Z}\\[.05in]
        & = & \j-\chi_{\cO_1}+(q^2-q-1)\chi_{\cO_2}-\chi_{J_1}-\chi_{J_2}+(q-1)\chi_W+(q-2)\chi_{Z}.
\end{array}
\]

We will now calculate $\chi_{\cO_1}A_4$, using the fact that the sum of the adjacency matrices is the all-ones matrix $J$: 
\begin{align*}
\chi_{\cO_1}A_4 & = \chi_{\cO_1}J - (\chi_{\cO_1}I +\chi_{\cO_1}A_1 +\chi_{\cO_1}A_2+\chi_{\cO_1}A_3 +\chi_{\cO_1}A_5)\\[.05in]
& = (q^2-1)(\j-\chi_{\cO_1}-\chi_{\cO_2}-\chi_{J_1})-(q-1)(\chi_{J_2}+\chi_V+2\chi_Z)-(2q-1)\chi_{W}.
\end{align*}
The same arguments work for the characteristic vector $\chi_{\cO_2}$.\qedhere
\end{proof}
\begin{corollary}\label{cor_2}
Let $v=\chi_{\cO_1}-\chi_{\cO_2}$. Then:
\begin{itemize}
\item[] $vA_1=-v+q(\chi_{J_1}-\chi_{J_2})$;
\item[] $vA_2= \chi_{J_1}-\chi_{J_2}$;
\item[] $vA_3=  q(q-1)v +(q^2-2q-1)(\chi_{J_1}-\chi_{J_2})$;
\item[] $vA_4= -(q^2-q)(\chi_{J_1}-\chi_{J_2})$;
\item[] $vA_5= -(q^2-q)v$.
\end{itemize}
\end{corollary}
\begin{proposition}\label{prop_4}
The dual degree set of $v=\chi_{\cO_1}-\chi_{\cO_2}$ is $\{1,5\}$.
\end{proposition}
\begin{proof}
By using Theorem \ref{prop_1},  we  express each idempotent matrix $E_j$, $j=0,\ldots,5$, of $\mathfrak X_l$ in
terms of adjacency matrices as $E_i =\frac{1}{q^5}\sum_{j=0}^{5}{\cQ(j,i)A_j}$, where $\cQ(j,i)$ is the $(j,i)-$entry of $\cQ$.
From Corollary \ref{cor_2}, we have:
{\small
\begin{align*}
q^5vE_0  =\, &  vJ = 0;\\[.05in]
q^5vE_1  =\, &  v\left[(q-1)(q+1)^2(I+A_2)+(q^2-q-1)(A_1+A_3)-(q+1)(A_4+A_5)\right]\\
        = \,& q^4(v+\chi_{J_1}-\chi_{J_2})\neq 0;\\[.05in]
q^5vE_2  =\,& q^3 v\left[(q-1)I+\frac{(q-1)}{q+1}A_1-A_2-\frac{1}{q+1}A_3\right]=0;\\[.05in]
q^5vE_3  =\,&  q\,v\left[(q-1)^2(q+1)(I+A_2)-(q-1)(A_1+A_3)+A_4+A_5\right]=0;\\[.05in]
q^5vE_4  =\,&  \frac12 q^2v\left[(q-1)^3I-\frac{(q-1)^2}{(q+1)}A_1-(q-1)^2A_2+\frac{(q-1)}{(q+1)}A_3-A_4+(q-1)A_5\right]=0;\\[.05in]
q^5vE_5  =\, &  \frac12q^2v\left[(q+1)(q-1)^2I-(q-1)(A_1+A_5)-(q^2-1)A_2+A_3+A_4\right]\\[.05in]
         =\, & -q^2(\chi_{J_1}-\chi_{J_2})\neq 0.\qedhere
\end{align*}}
\end{proof}
Fix a totally singular line $l$ in $Q^-(5,q)$. For any given $B$ on $l$, two vectors are associated with each $\cU_{p_1,p_2}=\cO_1\cup \cO_2$ constructed on $(B,l)$, namely, $v=\chi_{\cO_1}-\chi_{\cO_2}$ and $-v=\chi_{\cO_2}-\chi_{\cO_1}$. Let $\cV_l$ be the set of all such vectors as $B$ varies on $l$.

\hspace{1cm}
\begin{lemma}\leavevmode\label{basicstuff}
\begin{enumerate}[(a)]
\item The number of $\cU_{p_1,p_2}$ constructed on the flag $(B,l)$  is $\frac{q+1}{2}q^3(q^2-1)$.
\item The number of $\cU_{p_1,p_2}$  constructed on the flag $(B,l)$  sharing a fixed line disjoint from $l$ is $(q+1)(q^2-1)$.
\end{enumerate}
\end{lemma}
\begin{proof}
To prove (a),  by using the polarity associated with $Q^-(5,q)$, it suffices to count the number of non-singular points in $B^\perp\setminus l^\perp$, for all $B\in l$. 
Secondly, (b) follows from the standard double counting of the pairs $(\cU_{p_1,p_2},m)$, with $m\in \cU_{p_1,p_2}$,  by considering (a) and the fact that the number of lines of each $\cU_{p_1,p_2}$ is $2q^2$.
\end{proof}
\begin{proposition}\label{sizeV}
 The size of $\cV_l$ is $\dim(V_1\perp V_5)=q^3(q-1)(q+1)^2$. 
\end{proposition}
\begin{proof}
This follows from Lemma \ref{basicstuff}, and taking into account that  $v=\chi_{\cO_1}-\chi_{\cO_2}\neq -v$. 
\end{proof}
Each of the minimal idempotents $E_i$, $i=0,\ldots,5$, of $\mathfrak X_l$ projects onto a common eigenspace $V_i$ of the adjacency matrices of the scheme. The vector space $\bR^{|\cX|}$, endowed with the standard inner product $\cdot\ $,  decomposes as $V_0\perp\cdots\perp V_5$, and a basis for it is the set of the characteristic vectors $\chi_m$ with $m\in\cX$. As usual,  $V_0$ is the space spanned by the all-ones vector $\j$.  Therefore, the set $\{\chi_mE_i: m\in\cX\}$ forms a basis for $V_i$, for $0\le i\le 5$, that is, $V_i=\mathrm{row}(E_i)$.

\begin{proposition}\label{prop_5}
 $\cV_l$  spans  $V_1	\perp V_5$.
\end{proposition}

\begin{proof}
Let $A$ be the matrix whose rows are the vectors $\chi_{\cO_1}-\chi_{\cO_2}$ in $\cV_l$, and  columns are indexed by  the elements of the scheme.   Let $M=A^\top A$. Note that it consists of the standard scalar products of columns of $A$.

For any line $m$, $\m$ will denote the column of $A$ pertaining to $m$. Index the elements of $\cV_l$ by $v_i$, where $i\in\{1,\ldots, q^3(q-1)(q+1)^2\}$.  Then, $\m_i=(v_i)_m$ and, by writing $v_i=\chi_{\cO_1}-\chi_{\cO_2}$, we have $\m_i=1$ if $m$ lies in $\cO_1$, $\m_i=-1$ if $m$ lies in $\cO_2$, $\m_i=0$ otherwise.

First we calculate what the diagonal entries of $M$ are. Note that  $\m\cdot \m=\sum_i{\m_i^2}$ equals the number of elements of $\cV_l$ whose support contains the line $m$. By using the standard double counting argument on pairs $(\cO_1,m)$, with $m\in\cO_1$, we get  $\m\cdot\m=2 (q-1) (q+1)^2$.

Now suppose $n$ is a line disjoint from $l$, not equal to $m$. To evaluate  $\m\cdot\n$, we take into account the equalities
\begin{equation}\label{eq_20}
    \m_i\n_i=\begin{cases}
    1& \text{if }\ \ m,n\in\cO_1\ \ \ \text{or }\ \ m,n\in\cO_2\\
    -1& \text{if }\ \ m\in\cO_1, n\in\cO_2\ \ \ \text{or }\ \ m\in\cO_2, n\in\cO_1\\
    0&\text{otherwise}
    \end{cases}
\end{equation}
and how $m$ and $n$ are related in the association scheme. We will use the calculations done in the proof of Lemma \ref{lem_11}. 

Assume $(m,n)\in R_1'$.   We first count in two different ways the number of triples $(\cO_1,m,n)$ with $m,n\in\cO_1$. We obtain
\[
c_1\eta_1=(q-1)(q-1) (q+1)^2,
\]
where $c_1$ is the number of the sets of type $\cO_1$ containing both $m$ and $n$. Hence, $c_1=q-1$. Similarly, for $m,n\in\cO_2$. 

We now count in different ways the number of triples $(\cO_1,m,n)$ with $m\in\cO_1$ and $n\in\cO_2$. It follows that 
\[
c_2\eta_1=q(q-1) (q+1)^2,
\]
where $c_2$ is the number of the sets of type $\cO_1\cup\cO_2$ such that $m\in\cO_1$ and $n\in\cO_2$. Hence, $c_2=q$. Similarly, for $m\in\cO_2$ and $n\in\cO_1$. This yields $\m\cdot\n=2(q-1)-2q=-2$.
 
Assume $(m,n)\in R_2'$. Consider the triples $(\cO_1,m,n)$ with $m,n\in\cO_1$. From the proof of Lemma \ref{lem_11}, we see that the number of such triples is zero. Similarly, for all other cases in \eqref{eq_20}. Hence, $\m\cdot\n=0$.

Assume $(m,n)\in R_3'$. We first count in different ways the number of triples $(\cO_1,m,n)$ with $m,n\in\cO_1$. We obtain
\[
c_3\eta_3=q(q-1)(q-1)(q+1)^2,
\]
where $c_3$ is the number of the sets of type $\cO_1$ containing both $m$ and $n$. Hence, $c_3=q$. Similarly, for $m,n\in\cO_2$. 

We now count in different ways the number of triples $(\cO_1,m,n)$ with $m\in\cO_1$ and $n\in\cO_2$. From the proof of Lemma \ref{lem_11}, this number is zero. Hence, $\m\cdot\n=2q$.

Assume $(m,n)\in R_4'$. Consider the triples $(\cO_1,m,n)$ with $m,n\in\cO_1$. From the proof of Lemma \ref{lem_11}, we see that the number of such triples is zero. Similarly, for all other cases in (\ref{eq_20}). Hence, $\m\cdot\n=0$.

Assume $(m,n)\in R_5'$. Then, the number of triples $(\cO_1,m,n)$ with $m,n\in\cO_1$ is zero. Similarly for $m,n\in\cO_2$.

We now count in different ways the number of triples $(\cO_1,m,n)$ with $m\in\cO_1$ and $n\in\cO_2$. It follows that 
\[
c_5\eta_5=(q^2-q)(q-1)(q+1)^2,
\]
where $c_5$ is the number of the sets of type $\cO_1\cup\cO_2$ such that $m\in\cO_1$ and $n\in\cO_2$. Hence, $c_5=q+1$. Similarly, for $m\in\cO_2$ and $n\in\cO_1$. This yields $\m\cdot\n=-2(q+1)$.
 
Therefore,
\[
M=2 \left( (q-1)  (q+1)^2 I- A_1+qA_3-(q+1)A_5\right)
\]
and, from the first eigenmatrix $\cP$, we see that
\begin{equation*}\label{eq_21}
M=2q^2(q^2E_1 + 2(q+1)E_5).
\end{equation*}
It is well known (c.f., \cite[Eq. (2.10)]{del}) that there exists an orthogonal matrix $U$ which simultaneously diagonalises each of the minimal idempotents of the scheme, i.e.,
\[
U^{-1}E_i U =\diag(0,\ldots,0,\underbrace{1,\ldots,1}_\text{$\dim V_i$},0,\ldots,0),
\]
for $i=0,\ldots,5$. This implies  $M$ itself takes a diagonal form with respect to the basis of the eigenvectors of $E_i$,  so that  
\[
\mathrm{row}(M)=\mathrm{row}(E_1)\perp \mathrm{row}(E_5)=V_1\perp V_5.
\]
 Therefore, $V_1\perp V_5=\mathrm{row}(M) \le \mathrm{row}(A)=\langle \cV_l \rangle$.  By Proposition \ref{prop_4}, $V_1\perp V_5=\langle \cV_l \rangle$.
\end{proof}
 \begin{theorem}\label{Ucharacterization1}
Let $\cS$ be a pseudo-oval of $Q^-(5,q)$. Then, the following are equivalent:
\begin{itemize}
\item[(a)] $\cS$ is a pseudo-conic;
\item[(b)] for any $l$ in $\cS$ and $B$ in $l$, each set $\cU_{p_1,p_2}$ constructed on $(B,l)$ meets $\cS\backslash\{l\}$ in $0$ or $2$ elements.
 \end{itemize}
 \end{theorem}
 \begin{proof}
By Theorem \ref{lem_13}, $\cS$ is a pseudo-conic if and only if
$\cS'=\cS\setminus\{l\}$ is a $\{1,5\}-$design of $\mathfrak X_l$, for any $l$ in $\cS$. 
Fix $l\in\cS$. By Proposition \ref{prop_5}, $\cV_l$ spans $V_1\perp V_5$. Hence, $\cS'$ is a $\{1,5\}-$design of  $\frak X_l$ if and only if $\chi_{\cS'}\cdot v=0$, for all $v\in \cV_l$. On the other hand,
\[
\chi_{\cS'}\cdot v=\chi_{\cS'}\cdot \chi_{\cO_1}-\chi_{\cS'}\cdot \chi_{\cO_2}=|\cS\cap \cO_1|-|\cS\cap \cO_2|.
\]
Since $\cS$ is a pseudo-oval,  we have $|\cS\cap \cO_1|,|\cS\cap \cO_2| \le 1$. Furthermore, 
\[
|\cS\cap (\cO_1\cup\cO_2)|=|\cS\cap \cO_1|+|\cS\cap \cO_2|,
\]
because  $\cO_1$ and $\cO_2$ are disjoint sets of lines.
Hence, $\chi_{\cS'}\cdot v=0$, for all $v\in \cV_l$, if and only if each $\cU_{p_1,p_2}$ meets $\cS'$ in $0$ or $2$ elements.
 \end{proof}
 Theorem \ref{Ucharacterization1} and the following proposition provide an additional way to characterise pseudo-conics.

\begin{proposition}\label{Ucharacterization2}
Let $\cS$ be a pseudo-oval of $Q^-(5,q)$ and $l\in \cS$. Let $A$ be the average
number of $\cU_{p_1,p_2}$ over all flags $(B,l)$ containing two distinct elements of $\cS\setminus\{l\}$.
Then, $A=q+1$ if and only if each $\cU_{p_1,p_2}$ meets $\cS\backslash\{l\}$ in  $0$ or $2$ elements.
\end{proposition}

\begin{proof}
Let $\mu_i$ be the number of lines of $\cS'=\cS\setminus\{l\}$ contained in the $i-$th set of type $\cU_{p_1,p_2}$.
 We count in two ways the number of pairs $(m,\cU_{p_1,p_2})$ such that $m\in\cU_{p_1,p_2}\cap\cS'$. From Lemma \ref{basicstuff}(b), this number is
\[
\sum_i \mu_i = |\cS'|(q+1)(q^2-1) = q^2(q+1)(q^2-1).
\] 
If we double count triples $(m,n,\cU_{p_1,p_2})$ where $m$ and $n$ are distinct elements of $\cS'$ lying in $\cU_{p_1,p_2}$, we see that 
\[
\sum_i\mu_i(\mu_i-1)=\sum_{m\in \cS'}\sum_{n\in \cS'\backslash\{m\}}|\{ \cU_{p_1,p_2} \colon m,n \in \cU_{p_1,p_2}\}|
=q^2(q^2-1) A.
\]
Hence, 
\[
\sum_i \mu_i(2-\mu_i)=\sum_i \mu_i-\sum_i \mu_i(\mu_i-1)=q^2(q^2-1)(q+1-A).
\]
Therefore, $A=q+1$ precisely when each $\mu_i$ is $0$ or $2$.
\end{proof}

\section{Concluding remarks}\label{sec_7}
In \cite{bokp}, Theorem 5.1 shows that  any special set of $H(3,9)$ is of CP-type, or dually, any pseudo-oval in $Q^-(5,3)$ is a pseudo-conic. By using \textsf{GAP} and the mixed integer linear programming software \textsf{Gurobi} \cite{gurobi}, we explored the case $q=5$ and $q=7$. Indeed, the theory developed in this paper aided in the design of the computation.

We look at a given pseudo-oval as a set $\cS$ of lines of $Q^-(5,q)$ such that every non-degenerate hyperplane contains $0$ or $2$ elements of $\cS$ by Proposition \ref{lem_18}. As we have done throughout this paper, we let $l$ be a fixed line of $Q^-(5,q)$. Let $M$ be the incidence matrix with rows indexed by lines of $Q^-(5,q)$ disjoint from $l$, and columns indexed by the non-degenerate hyperplanes not containing $l$. Then, we are seeking a solution to
\begin{equation}\label{tactical}
\bvec{x}M=2\bvec{y},
\end{equation}
where $\bvec{x}=(x_1,\ldots,x_{q^5})$ and  $\bvec{y}=(y_1,\ldots,y_{q^5-q^3})$, with $x_i,y_i\in\{0,1\}$, and $\sum{x_i}=q^2$. In fact, $\bvec{x}$ will be the characteristic vector for $\cS\setminus \{l\}$, with $\cS$ a pseudo-oval in $Q^-(5,q)$, and $\bvec{y}$ will be the characteristic vector for the set of non-degenerate hyperplanes not containing $l$, sharing two elements with $\cS\setminus\{l\}$.

There are a variety of approaches to solving equations such as \eqref{tactical}. In particular, the system of equations can be viewed either as an {\em integer linear program} or as a {\em constraint satisfaction problem}. We used the software \textsf{Gurobi}  for this problem. 

A linear program attempts to find values for variables $x_1, x_2, \ldots, x_n$ that maximise (or minimise) a linear objective function subject to linear constraints. An {\em integer linear program}, for short {\em integer program}, is a linear program with the additional restriction that the variables must take integral values. Solving \eqref{tactical} does not involve any maximising or minimising, so the objective function can be taken to be a constant, say 0. Then, any feasible solution to the following integer program yields a set of lines with the property:
\begin{equation}\label{linpro}
	\begin{array}{lrcl}
	\textrm{Maximise:}   & 0 &      &   \\
	\textrm{subject to:} & \bvec{x}B- 2\bvec{y}   & =    & 0\\
			& \sum_i x_i &=& q^2 \\
			    & x_i,y_j     & \in & \{0, 1\}.
	\end{array}
\end{equation}
There is one more ingredient we need to take into account. For a fixed set $U=\cU_{p_1,p_2}$, let $\bvec{u}$ be the characteristic vector for it. We assume  that the set $U$ meets $\cS\setminus\{l\}$ in precisely 1 element. This adds the linear constraint $\sum{u_i x_i} = 1$. For $q=3,5,7$, we found that the linear program \eqref{linpro} is infeasible\footnote{We would like to thank Jesse Lansdown for verifying our computations.} for each $\cU_{p_1,p_2}$. Therefore, in these cases, every set $\cS\setminus\{l\}$ is forced to meet the sets $\cU_{p_1,p_2}$   in 0 or 2 elements, and so  every $\cS$ is a pseudo-conic by Theorem \ref{Ucharacterization1}.

The above computational results suggest the following conjecture:
\begin{conjecture}
Every pseudo-oval in $Q^-(5,q)$ is a pseudo-conic, for any $q$ (odd). 
\end{conjecture}
Results \ref{lem_18} and \ref{Ucharacterization1}  allow us to state the above conjecture  as follows:

\begin{conjecture}\label{conj_2}
Let $\cS$ be a set of $q^2+1$ lines of $Q^-(5,q)$, $q$ odd, such that every non-degenerate hyperplane contains 0 or 2 elements of $\cS$. Then, each $\cU_{p_1,p_2}$ meets $\mathcal{S}\backslash\{l\}$ in 0 or 2 elements, for every $l\in\cS$. 
\end{conjecture}

\end{document}